\documentclass{amsart}%
\usepackage{amsfonts}
\usepackage{amsmath}
\usepackage{amssymb}
\usepackage{graphicx}%
\newtheorem{theorem}{Theorem}
\theoremstyle{plain}

\newtheorem{corollary}{Corollary}

\newtheorem{definition}{Definition}

\newtheorem{lemma}{Lemma}
\newtheorem{notation}{Notation}

\newtheorem{proposition}{Proposition}
\newtheorem{remark}{Remark}

\numberwithin{equation}{section}
\begin{document}
\title[Symmetric and Antisymmetric Polynomials]{Symmetric and Antisymmetric Vector-valued Jack Polynomials}
\author{Charles F. Dunkl}
\address{Department of Mathematics, University of Virginia\\
Charlottesville, VA 22904-4137, US}
\email{cfd5z@virginia.edu}
\urladdr{http://people.virginia.edu/\symbol{126}cfd5z}
\date{28 January 2010}

\begin{abstract}
Polynomials with values in an irreducible module of the symmetric group can be
given the structure of a module for the rational Cherednik algebra, called a
standard module. This algebra has one free parameter and is generated by
differential-difference (\textquotedblleft Dunkl\textquotedblright) operators,
multiplication by coordinate functions and the group algebra. By specializing
Griffeth's (arXiv:0707.0251) results for the G(r,p,n) setting, one obtains
norm formulae for symmetric and antisymmetric polynomials in the standard
module. Such polynomials of minimum degree have norms which involve
hook-lengths and generalize the norm of the alternating polynomial.

\end{abstract}
\maketitle

\section{Introduction}

Hook-lengths of nodes in Young tableaux appear in a variety of different
settings. Griffeth \cite{G} introduced Jack polynomials whose values lie in
irreducible modules of the complex reflection group family $G\left(
r,p,N\right)  $. This class of polynomials forms an orthogonal basis for the
associated standard module of the rational Cherednik algebra. In this paper we
specialize his results to the symmetric group and show how the norms of two
special symmetric and antisymmetric polynomials in the standard module depend
on the hook-lengths of the partition associated to the representation.

For $N\geq2,x=\left(  x_{1},\ldots,x_{N}\right)  \in%
%TCIMACRO{\U{211d} }%
%BeginExpansion
\mathbb{R}
%EndExpansion
^{N}$ and let $\mathbb{N}_{0}:=\left\{  0,1,2,3,\ldots\right\}  $. For
$a,b\in\mathbb{N}_{0}$ and $a\leq b$ let $\left[  a,b\right]  =\left\{
a,a+1,\ldots,b\right\}  $ (an interval of integers). The cardinality of a set
$E$ is denoted by $\#E$. For $\alpha\in\mathbb{N}_{0}^{N}$ (a
\textit{composition}) let $\left\vert \alpha\right\vert :=\sum_{i=1}^{N}%
\alpha_{i}$, $x^{\alpha}:=\prod_{i=1}^{N}x_{i}^{\alpha_{i}}$, a monomial of
degree $\left\vert \alpha\right\vert $. The spaces of polynomials,
respectively homogeneous, polynomials are
\begin{align*}
\mathcal{P}  &  :=\mathrm{span}_{\mathbb{F}}\left\{  x^{\alpha}:\alpha
\in\mathbb{N}_{0}^{N}\right\}  ,\\
\mathcal{P}_{n}  &  :=\mathrm{span}_{\mathbb{F}}\left\{  x^{\alpha}:\alpha
\in\mathbb{N}_{0}^{N},\left\vert \alpha\right\vert =n\right\}  ,~n\in%
%TCIMACRO{\U{2115} }%
%BeginExpansion
\mathbb{N}
%EndExpansion
_{0},
\end{align*}
where $\mathbb{F}$ is a field $\supset\mathbb{Q}$. Consider the symmetric
group $\mathcal{S}_{N}$ as the group of permutations of $\left[  1,N\right]
$. The group acts on polynomials by linear extension of $\left(  xw\right)
_{i}=x_{w\left(  i\right)  },w\in\mathcal{S}_{N},1\leq i\leq N$, that is,
$wf\left(  x\right)  :=f\left(  xw\right)  ,f\in\mathcal{P}$. For $\alpha
\in\mathbb{N}_{0}^{N}$ let $\left(  w\alpha\right)  _{i}=\alpha_{w^{-1}\left(
i\right)  }$, then $w\left(  x^{\alpha}\right)  =x^{w\alpha}$. Also
$\mathcal{S}_{N}$ is a finite reflection group whose reflections are the
transpositions $\left(  i,j\right)  $; $x\left(  i,j\right)  =\left(
\ldots,\overset{i}{x}_{j},\ldots,\overset{j}{x}_{i},\ldots\right)  $. The
simple reflections $s_{i}:=\left(  i,i+1\right)  ,1\leq i<N$, generate
$\mathcal{S}_{N}$.

Say $\lambda\in\mathbb{N}_{0}^{N}$ is a partition if $\lambda_{i}\geq
\lambda_{i+1}$ for all $i$. Denote the set of partitions by $\mathbb{N}%
_{0}^{N,+}$. Suppose $\tau$ is a partition of $N$, that is, $\left\vert
\tau\right\vert =N$; then there is an associated Ferrers diagram, namely the
set of lattice points $\left\{  \left(  i,j\right)  \in\mathbb{N}_{0}%
^{2}:1\leq i\leq\ell\left(  \tau\right)  ,1\leq j\leq\tau_{i}\right\}  $, also
denoted by $\tau$; the \textit{length}\textbf{ }of $\tau$ is $\ell\left(
\tau\right)  :=\max\left\{  i:\tau_{i}>0\right\}  $. The conjugate partition
$\tau^{\prime}$ is the partition whose diagram is the transpose of the diagram
of $\tau$ (that is, $\tau_{m}^{\prime}=\#\left\{  i:\tau_{i}\geq m\right\}  $.
For a node (or point) $\left(  i,j\right)  \in\tau$ the \textit{arm-length} is
$\mathrm{arm}\left(  i,j\right)  :=\tau_{i}-j$, the \textit{leg-length} is
$\mathrm{leg}\left(  i,j\right)  :=\tau_{j}^{\prime}-i$, and the
\textit{hook-length} is $h\left(  i,j\right)  :=\mathrm{arm}\left(
i,j\right)  +\mathrm{leg}\left(  i,j\right)  +1$. We will use $\mathrm{arm}%
\left(  i,j;\tau\right)  $ etc. if it is necessary to specify the partition.

To each partition $\tau$ of $N$ there is an associated irreducible
$\mathcal{S}_{N}$-module $V_{\tau}$. We analyze the space $M\left(
\tau\right)  $ of $V_{\tau}$-valued polynomials under the action of
differential-difference (\textquotedblleft Dunkl\textquotedblright) operators.
There is a canonical symmetric bilinear (the contravariant) form $\left\langle
\cdot,\cdot\right\rangle $ on this space. We will construct distinguished
polynomials $f_{\tau}^{s},f_{\tau}^{a}\in M\left(  \tau\right)  $, with
$f_{\tau}^{s}$ being symmetric and $f_{\tau}^{a}$ being antisymmetric, such
that
\begin{align*}
\left\langle f_{\tau}^{s},f_{\tau}^{s}\right\rangle  &  =c_{0}\prod
\limits_{\left(  i,j\right)  \in\tau}\left(  1-h\left(  i,j\right)
\kappa\right)  _{\mathrm{leg}\left(  i,j\right)  },\\
\left\langle f_{\tau}^{a},f_{\tau}^{a}\right\rangle  &  =c_{1}\prod
\limits_{\left(  i,j\right)  \in\tau}\left(  1+h\left(  i,j\right)
\kappa\right)  _{\mathrm{arm}\left(  i,j\right)  };
\end{align*}
and $c_{0},c_{1}\in\mathbb{Q}$ are constants depending on $\tau$ (the
\textit{Pochhammer symbol} is $\left(  t\right)  _{n}:=\prod_{i=1}^{n}\left(
t+i-1\right)  ,n\in\mathbb{N}_{0}$). This result generalizes the situation of
the trivial representation of $\mathcal{S}_{N}$; in this case $\tau=\left(
N\right)  ,f_{\tau}^{s}=1,f_{\tau}^{a}=\prod\limits_{1\leq i<j\leq N}\left(
x_{i}-x_{j}\right)  $ and $\left\langle f_{\tau}^{a},f_{\tau}^{a}\right\rangle
=c_{1}\prod_{i=2}^{N}\left(  1+i\kappa\right)  _{i-1}$.

Section 2 collects the needed information about representations of
$\mathcal{S}_{N}$. The Dunkl operators and their action on monomials are
discussed in Section 3. The (nonsymmetric) Jack polynomials are constructed in
Section 4; this material is the specialization of Griffeth's results for
$G\left(  r,p,N\right)  $ to $\mathcal{S}_{N}=G\left(  1,1,N\right)  $. Our
main results on symmetric and antisymmetric polynomials are contained in
Section 5. There is the description of an orthogonal basis and the detailed
exposition of the special polynomials whose norms involve the hook-lengths. In
fact these are the symmetric and antisymmetric polynomials of minimum degree.

\section{Representations of $\mathcal{S}_{N}$}

Let $Y\left(  \tau\right)  $ be the set of \textit{reversed standard Young
tableaux} (RSYT) of shape $\tau$, namely, an assignment of the numbers
$\left\{  1,2,\ldots,N\right\}  $ to each node of $\tau$ such that entries
decrease in each row and in each column. The node of $T$ containing $i$ is
denoted $T\left(  i\right)  $ and the row and column of this node are denoted
by $\mathrm{rw}\left(  i,T\right)  ,\mathrm{cm}\left(  i,T\right)  $
respectively, $i\in\left[  1,N\right]  $. The \textit{content} of $T\left(
i\right)  $ is $c\left(  i,T\right)  :=\mathrm{cm}\left(  i,T\right)
-\mathrm{rw}\left(  i,T\right)  $. Thus $c\left(  N,T\right)  =0$ for each
$T$. The well-known hook-length formula asserts that $\#Y\left(  \tau\right)
=N!/\prod\limits_{\left(  i,j\right)  \in\tau}h\left(  i,j\right)  $.
Following Murphy \cite{M} define an action of $\mathcal{S}_{N}$ on the
$\#\left(  Y\left(  \tau\right)  \right)  $-dimensional vector space $V_{\tau
}:=\mathrm{span}_{\mathbb{F}}\left\{  v_{T}:T\in Y\left(  \tau\right)
\right\}  $ as follows:

\begin{proposition}
\label{siv}Suppose $T\in Y\left(  \tau\right)  $ and $b_{i}\left(  T\right)
:=1/\left(  c\left(  i,T\right)  -c\left(  i+1,T\right)  \right)  $ for $1\leq
i<N$ then:\newline1) if $b_{i}\left(  T\right)  =1$ (when \thinspace
$\mathrm{rw}\left(  i,T\right)  =\mathrm{rw}\left(  i+1,T\right)  $) then
$s_{i}v_{T}=v_{T}$;\newline2) if $b_{i}\left(  T\right)  =-1$ (when
\thinspace$\mathrm{cm}\left(  i,T\right)  =\mathrm{cm}\left(  i+1,T\right)  $)
then $s_{i}v_{T}=-v_{T}$;\newline3) if $0<b_{i}\left(  T\right)  \leq\frac
{1}{2}$ (when $\mathrm{rw}\left(  i,T\right)  <\mathrm{rw}\left(
i+1,T\right)  $ and $\mathrm{cm}\left(  i,T\right)  >\mathrm{cm}\left(
i+1,T\right)  $then%
\begin{align*}
s_{i}v_{T}  &  =b_{i}\left(  T\right)  v_{T}+v_{s_{i}T},\\
s_{i}v_{s_{i}T}  &  =\left(  1-b_{i}\left(  T\right)  ^{2}\right)  v_{T}%
-b_{i}\left(  T\right)  v_{s_{i}T};
\end{align*}
\newline4) if $-\frac{1}{2}\leq b_{i}\left(  T\right)  <0$ (when
$\mathrm{rw}\left(  i,T\right)  >\mathrm{rw}\left(  i+1,T\right)  $ and
$\mathrm{cm}\left(  i,T\right)  <\mathrm{cm}\left(  i+1,T\right)  )$ then
$s_{i}v_{T}=b_{i}\left(  T\right)  v_{T}+\left(  1-b_{i}\left(  T\right)
^{2}\right)  v_{s_{i}T},$ and $s_{i}v_{s_{i}T}=v_{T}-b_{i}\left(  T\right)
v_{s_{i}T}$.\newline In cases (3) and (4) the tableau $s_{i}T$ is obtained by
interchanging the entries $i,i+1$. Furthermore in case (3) $s_{i}f_{0}=f_{0}$
and $s_{i}f_{1}=-f_{1}$ for
\begin{align*}
f_{0}  &  =\left(  b_{i}\left(  T\right)  +1\right)  v_{T}+v_{s_{i}T},\\
f_{1}  &  =\left(  b_{i}\left(  T\right)  -1\right)  v_{T}+v_{s_{i}T}.
\end{align*}

\end{proposition}

There is an ordering on tableaux such that $T>s_{i}T$ in case (4).

\begin{corollary}
\label{vsym}Let $f=\sum_{T\in Y\left(  \tau\right)  }k_{T}v_{T}$ with the
coefficients $k_{T}\in\mathbb{Q}$ and $s_{i}f=\pm f$ for some $i\in\left[
1,N-1\right]  $. Then\newline1) $T,s_{i}T\in Y\left(  \tau\right)  $ implies
$k_{s_{i}T}=rk_{T}$ for some $r\neq0;$\newline2) $s_{i}f=f$ and $\mathrm{cm}%
\left(  i,T\right)  =\mathrm{cm}\left(  i+1,T\right)  $ implies $k_{T}%
=0$;\newline3) $s_{i}f=-f$ and $\mathrm{rw}\left(  i,T\right)  =\mathrm{rw}%
\left(  i+1,T\right)  $ implies $k_{T}=0.$
\end{corollary}

Statement (1) means that $k_{s_{i}T}$ and $k_{T}$ are either both nonzero or
both zero.

\begin{definition}
The Jucys-Murphy elements (in the group algebra $\mathbb{Q}\mathcal{S}_{N}$)
are%
\[
\omega_{i}:=\sum_{j=i+1}^{N}\left(  i,j\right)  ,1\leq i\leq N.
\]

\end{definition}

There are commutation relations: $\omega_{i}\omega_{j}=\omega_{i}\omega_{j}$
for all $i,j$; $\omega_{i}s_{j}=s_{j}\omega_{i}$ for $j\neq i-1,i$;
$s_{i}\omega_{i}-\omega_{i+1}s_{i}=1$ (see Vershik and Okounkov \cite[Section
4]{OV} for the representations of the algebra generated by $\left\{
\omega_{i},\omega_{i+1},s_{i}\right\}  $). Murphy proved the following:

\begin{theorem}
Suppose $T\in Y\left(  \tau\right)  $ and $i\in\left[  1,N\right]  $ then
$\omega_{i}v_{T}=c\left(  i,T\right)  v_{T}.$
\end{theorem}

Let $\left\langle \cdot,\cdot\right\rangle _{0}$ be a $\mathcal{S}_{N}%
$-invariant positive-definite linear form on $V_{\tau}$, (the form is unique
up to a multiplicative constant) then each $\omega_{i}$ is self-adjoint and
hence the vectors $v_{T}$ are pairwise orthogonal, being eigenvectors with
different eigenvalues. Denote $\left\Vert v\right\Vert _{0}^{2}=\left\langle
v,v\right\rangle $ For given $T$ and $i$ as in case (4) we have $\left\Vert
v_{T}\right\Vert _{0}^{2}=b\left(  i,T\right)  ^{2}\left\Vert v_{T}\right\Vert
_{0}^{2}+\left\Vert v_{s_{i}T}\right\Vert _{0}^{2}$ (since $s_{i}$ is an
isometry) and thus $\left\Vert v_{s_{i}T}\right\Vert _{0}^{2}=\left(
1-b\left(  i,T\right)  ^{2}\right)  \left\Vert v_{T}\right\Vert _{0}^{2}$.
There is one formula for $\left\Vert v_{T}\right\Vert _{0}^{2}$ in \cite[Thm.
4.1]{M}. The following is based on the content vector of $T$ (that is,
$\left(  c\left(  1,T\right)  ,\ldots,c\left(  N,T\right)  \right)  ):$

\begin{definition}
For $T\in Y\left(  \tau\right)  $ let%
\[
\left\Vert v_{T}\right\Vert _{c}^{2}=\prod\limits_{1\leq i<j\leq N,~c\left(
i,T\right)  \leq c\left(  j,T\right)  -2}\frac{\left(  c\left(  i,T\right)
-c\left(  j,T\right)  \right)  ^{2}-1}{\left(  c\left(  i,T\right)  -c\left(
j,T\right)  \right)  ^{2}}.
\]

\end{definition}

\begin{lemma}
\label{gprod}Suppose $\left\{  g_{ij}\left(  T\right)  :1\leq i<j\leq
N\right\}  $ is a collection of functions on $Y\left(  \tau\right)  $ and
satisfy (1) $g_{ij}\left(  T\right)  =g_{ij}\left(  s_{m}T\right)  $ for all
$i,j$ with $\left\{  i,j\right\}  \cap\left\{  m,m+1\right\}  =\emptyset$, (2)
$g_{i,m}\left(  T\right)  =g_{i,m+1}\left(  s_{m}T\right)  $ and
$g_{i,m+1}\left(  T\right)  =g_{i,m}\left(  s_{m}T\right)  $ for $i<m$, (3)
$g_{m,j}\left(  T\right)  =g_{m+1,j}\left(  s_{m}T\right)  $ and
$g_{m+1,j}\left(  T\right)  =g_{m,j}\left(  s_{m}T\right)  $ for $j>m+1$ for
all $T\in Y\left(  \tau\right)  $ and $m\in\left[  1,N-1\right]  $ such that
$s_{m}T\in Y\left(  \tau\right)  ,$ then%
\[
\frac{\prod_{1\leq i<j\leq N}g_{ij}\left(  T\right)  }{\prod_{1\leq i<j\leq
N}g_{ij}\left(  s_{m}T\right)  }=\frac{g_{m,m+1}\left(  T\right)  }%
{g_{m,m+1}\left(  s_{m}T\right)  }.
\]

\end{lemma}

The proof is a straightforward calculation.

\begin{proposition}
Suppose $0<b_{i}\left(  T\right)  \leq\frac{1}{2}$ for $T\in Y\left(
\tau\right)  $ and some $i\in\left[  1,N-1\right]  $ then $\left\Vert
v_{s_{i}T}\right\Vert _{c}^{2}=\left(  1-b_{i}\left(  T\right)  ^{2}\right)
\left\Vert v_{T}\right\Vert _{c}^{2}$. Thus $\left\Vert \cdot\right\Vert _{c}$
is an $\mathcal{S}_{N}$-invariant norm.
\end{proposition}

\begin{proof}
By hypothesis $c\left(  i,T\right)  \geq c\left(  i+1,T\right)  +2$ and
$c\left(  i,s_{i}T\right)  =c\left(  i+1,T\right)  ,~c\left(  i+1,s_{i}%
T\right)  =c\left(  i,T\right)  $. In the ratio $\left\Vert v_{s_{i}%
T}\right\Vert _{c}^{2}/\left\Vert v_{T}\right\Vert _{c}^{2}$ all factors
except $\left(  1-\left(  \frac{1}{c\left(  i+1,T\right)  -c\left(
i,T\right)  }\right)  ^{2}\right)  $ in the numerator cancel out, by Lemma
\ref{gprod}.
\end{proof}

Henceforth we drop the subscript \textquotedblleft$c$\textquotedblright\ and
use \textquotedblleft$0\textquotedblright$ for the form. Next we consider
invariance properties for certain subgroups of $\mathcal{S}_{N}$. Specifically
these are the stabilizer subgroups of a monomial $x^{\lambda}$, where
$\lambda\in\mathbb{N}_{0}^{N,+}$.

\begin{definition}
For $1\leq a<b\leq N$ let $\mathcal{S}_{\left[  a,b\right]  }=\left\{
w\in\mathcal{S}_{N}:i\notin\left[  a,b\right]  \Longrightarrow w\left(
i\right)  =i\right\}  $, the subgroup of permutations of $\left[  a,b\right]
$, generated by$\left\{  s_{i}:a\leq i<b\right\}  $.
\end{definition}

We look for elements $f$ of $V_{\tau}$ which are symmetric or antisymmetric
for a group $\mathcal{S}_{\left[  a,b\right]  }$, or the equivalent
properties: $s_{i}f=f$, respectively, $s_{i}f=-f$, for $a\leq i<b$. Roughly,
start with some $v_{T}$ and analyze $\sum_{w\in\mathcal{S}_{\left[
a,b\right]  }}wv_{T}$, or $\sum_{w\in\mathcal{S}_{\left[  a,b\right]  }}%
$\textrm{sgn}$\left(  w\right)  ~wv_{T}$, expanded in the basis $\left\{
v_{S}:S\in Y\left(  \tau\right)  \right\}  $.

\begin{definition}
For $T\in Y\left(  \tau\right)  $ and a subgroup $H$ of $\mathcal{S}_{N}$ let
$V_{T}\left(  H\right)  =\mathrm{span}\left\{  wv_{T}:w\in H\right\}  $ and
let $Y\left(  T;H\right)  =\left\{  T^{\prime}\in Y\left(  \tau\right)
:v_{T^{\prime}}\in V_{T}\left(  H\right)  \right\}  $.
\end{definition}

In the case $H=\mathcal{S}_{\left[  a,b\right]  }$ there are two extremal
elements of $Y\left(  T;H\right)  $, namely $T_{0}$ with the property
$\mathrm{cm}\left(  i,T_{0}\right)  \geq\mathrm{cm}\left(  i+1,T_{0}\right)  $
for $a\leq i<b$, and $T_{1}$ with the property $\mathrm{rw}\left(
i,T_{0}\right)  \geq\mathrm{rw}\left(  i+1,T_{0}\right)  $ (it is possible
that $T_{0}=T_{1}$). To produce $T_{0}$ one applies a sequence of
transformations of type (4) (in Prop. \ref{siv}) (type (3) for $T_{1}$). If
$\mathrm{cm}\left(  i_{1},T\right)  =\mathrm{cm}\left(  i_{2},T\right)  $ for
some $i_{1},i_{2}\in\left[  a,b\right]  $ (suppose $i_{1}>i_{2}$ then any
entry $j$ in this column of $T$ between $i_{1}$ and $i_{2}$ has to satisfy
$i_{1}>j>i_{2}$) then $T_{0}$ has $\mathrm{cm}\left(  i,T_{0}\right)
=\mathrm{cm}\left(  i+1,T_{0}\right)  $ for some $i\in\left[  a,b-1\right]  $.
Similarly if $\mathrm{rw}\left(  i_{1},T\right)  =\mathrm{rw}\left(
i_{2},T\right)  $ for some $i_{1},i_{2}\in\left[  a,b\right]  \ $then $T_{1}$
has $\mathrm{rw}\left(  i,T_{1}\right)  =\mathrm{rw}\left(  i+1,T_{1}\right)
$ for some $i\in\left[  a,b-1\right]  $.

First consider the invariant (symmetric) situation. Corollary \ref{vsym} and
the properties of $T_{0}$ imply the following necessary condition for
$V_{T}\left(  \mathcal{S}_{\left[  a,b\right]  }\right)  $ to contain a
nontrivial $\mathcal{S}_{\left[  a,b\right]  }$-invariant.

Say $T$ satisfies condition $\left[  a,b\right]  _{\mathrm{cm}}$ if $a\leq
i<j\leq m$ implies $\mathrm{cm}\left(  i,T\right)  \neq\mathrm{cm}\left(
j,T\right)  $ (the entries $a,a+1,\ldots,b$ are in distinct columns of $T$).
Fix some $T$ satisfying this condition and consider the subspace $V_{T}\left(
\mathcal{S}_{\left[  a,b\right]  }\right)  $. Let $T_{0}\in Y\left(
T;\mathcal{S}_{\left[  a,b\right]  }\right)  $ satisfy $\mathrm{cm}\left(
i,T\right)  >\mathrm{cm}\left(  j,T\right)  $ for $a\leq i<j\leq b$ (equality
is ruled out by hypothesis). It is possible that $i$ and $i+1$ are in the same
row of $T_{0}$ for some $i\in\left[  a,b\right]  $ (in which case $\#Y\left(
T;\mathcal{S}_{\left[  a,b\right]  }\right)  <\left(  b-a+1\right)
!=\#\mathcal{S}_{\left[  a,b\right]  }$). For $a\leq i<b$ we have
$\mathrm{rw}\left(  i,T_{0}\right)  \leq\mathrm{rw}\left(  i+1,T_{0}\right)
$, thus $a\leq i<j\leq b$ implies $c\left(  j,T_{0}\right)  -c\left(
i,T_{0}\right)  \leq-2$ or $j=i+1$ and $\mathrm{rw}\left(  i+1,T_{0}\right)
=\mathrm{rw}\left(  i,T_{0}\right)  $; indeed suppose the latter condition
does not hold then if $j>i+1$%
\begin{align*}
c\left(  j,T_{0}\right)  -c\left(  i,T_{0}\right)   &  =\left(  \mathrm{cm}%
\left(  j,T_{0}\right)  -\mathrm{cm}\left(  i,T_{0}\right)  \right)  +\left(
\mathrm{rw}\left(  i,T_{0}\right)  -\mathrm{rw}\left(  j,T_{0}\right)  \right)
\\
&  \leq\mathrm{cm}\left(  j,T_{0}\right)  -\mathrm{cm}\left(  i,T_{0}\right)
\leq i-j\leq-2,
\end{align*}
or $j=i+1$ and
\begin{align*}
c\left(  i+1,T_{0}\right)  -c\left(  i,T_{0}\right)   &  =\left(
\mathrm{cm}\left(  i+1,T_{0}\right)  -\mathrm{cm}\left(  i,T_{0}\right)
\right)  +\left(  \mathrm{rw}\left(  i,T_{0}\right)  -\mathrm{rw}\left(
i+1,T_{0}\right)  \right) \\
&  \leq-1-1=-2.
\end{align*}

\begin{definition}
Suppose $T\in Y\left(  \tau\right)  $ satisfies condition $\left[  a,b\right]
_{\mathrm{cm}}$ then let%
\[
P_{0}\left(  T;a,b\right)  =\prod\limits_{a\leq i<j\leq b,~\mathrm{cm}\left(
i,T\right)  <\mathrm{cm}\left(  j,T\right)  }\frac{c\left(  j,T\right)
-c\left(  i,T\right)  }{1+c\left(  j,T\right)  -c\left(  i,T\right)  }.
\]

\end{definition}

The denominator can not vanish, for suppose $i<j,\mathrm{cm}\left(
i,T\right)  <\mathrm{cm}\left(  j,T\right)  $, and $T\left(  i\right)
=T_{0}\left(  i_{1}\right)  ,T\left(  j\right)  =T_{0}\left(  i_{2}\right)  $
with $i_{1}<i_{2}$ (this follows from $\mathrm{cm}\left(  i_{2},T\right)
<\mathrm{cm}\left(  i_{1},T\right)  $) then $c\left(  i,T\right)  -c\left(
j,T\right)  =c\left(  i_{2},T_{0}\right)  -c\left(  i_{1},T_{0}\right)
\leq-2$ , and $\mathrm{rw}\left(  i,T\right)  =\mathrm{rw}\left(  i_{2}%
,T_{0}\right)  \neq\mathrm{rw}\left(  j,T\right)  =\mathrm{rw}\left(
i_{1},T_{0}\right)  $. For notational convenience we use the fact $Y\left(
T;\mathcal{S}_{\left[  a,b\right]  }\right)  =Y\left(  T_{0};\mathcal{S}%
_{\left[  a,b\right]  }\right)  $ (and let $T$ be variable, henceforth).

\begin{proposition}
\label{symv}Let $f=\sum_{T\in Y\left(  T_{0};\mathcal{S}_{\left[  a,b\right]
}\right)  }P_{0}\left(  T;a,b\right)  v_{T}$ then $wf=f$ for all
$w\in\mathcal{S}_{\left[  a,b\right]  }$.
\end{proposition}

\begin{proof}
Suppose $a\leq i<b$ then let $A=\left\{  T\in Y\left(  T_{0};\mathcal{S}%
_{\left[  a,b\right]  }\right)  :\mathrm{rw}\left(  i,T\right)  =\mathrm{rw}%
\left(  i+1,T\right)  \right\}  $ and $B=\left\{  T\in Y\left(  T_{0}%
;\mathcal{S}_{\left[  a,b\right]  }\right)  :\mathrm{rw}\left(  i,T\right)
<\mathrm{rw}\left(  i+1,T\right)  \right\}  $. Then%
\[
f=\sum_{T\in A}P_{0}\left(  T;a,b\right)  v_{T}+\sum_{T\in B}\left(
P_{0}\left(  T;a,b\right)  v_{T}+P_{0}\left(  s_{i}T;a,b\right)  v_{s_{i}%
T}\right)  .
\]
Fix $T\in B$ and compute $P_{0}\left(  T;a,b\right)  /P_{0}\left(
s_{i}T;a,b\right)  $ using Lemma \ref{gprod}; set $g_{mn}\left(  T\right)  =1$
if $\mathrm{cm}\left(  m,T\right)  \geq\mathrm{cm}\left(  n,T\right)  $ and
$g_{mn}\left(  T\right)  =\frac{c\left(  n,T\right)  -c\left(  m,T\right)
}{1+c\left(  n,T\right)  -c\left(  m,T\right)  }$ if $\mathrm{cm}\left(
m,T\right)  <\mathrm{cm}\left(  n,T\right)  $. Then $g_{i,i+1}\left(
T\right)  =1$ and $g_{i,i+1}\left(  s_{i}T\right)  =\frac{c\left(
i+1,s_{i}T\right)  -c\left(  i,s_{i}T\right)  }{1+c\left(  i+1,s_{i}T\right)
-c\left(  i,s_{i}T\right)  }=\frac{1}{1-b_{i}\left(  s_{i}T\right)  }=\frac
{1}{1+b_{i}\left(  T\right)  }$. Thus $P_{0}\left(  T;a,b\right)
/P_{0}\left(  s_{i}T;a,b\right)  =1+b_{i}\left(  T\right)  $ and $s_{i}f=f$ by
Proposition \ref{siv}.
\end{proof}

\begin{corollary}
\label{svnorm}Let $n_{0}=\#\left\{  w\in\mathcal{S}_{\left[  a,b\right]
}:wv_{T_{0}}=v_{T_{0}}\right\}  $, then%
\[
\left\Vert f\right\Vert _{0}^{2}=\frac{\left(  b-a\right)  !}{n_{0}}%
P_{0}\left(  T_{1};a,b\right)  \left\Vert v_{T_{0}}\right\Vert _{0}^{2}.
\]

\end{corollary}

\begin{proof}
If $T,T^{\prime}\in Y\left(  \tau\right)  $ and $T^{\prime}$ is obtained from
$T$ by a sequence of steps of type (3) in Proposition \ref{siv} then
$T^{\prime}=wT$ for some $w\in\mathcal{S}_{N}$ and $v_{T^{\prime}}=wv_{T}%
+\sum_{j}b_{j}v_{S_{j}}$, where $b_{j}\in\mathbb{Q}$ and $\left[
S_{1}=T,S_{2},\ldots\right]  $ is the list of intermediate steps. Let
$f_{1}=\sum\limits_{w\in\mathcal{S}_{\left[  a,b\right]  }}wv_{T_{0}}$ thus
$f_{1}=cf$ for some constant $c$. In the expansion of $f_{1}$ in the basis
$\left\{  v_{T}:T\in Y\left(  T_{0};\mathcal{S}_{\left[  a,b\right]  }\right)
\right\}  $ the coefficient of $v_{T_{1}}$ is $n_{0}$, because $T_{0},T_{1}$
have the property described above and $v_{T_{1}}$ is extremal in $Y\left(
T_{0};\mathcal{S}_{\left[  a,b\right]  }\right)  $ (heuristically the
\textquotedblleft bubble sort\textquotedblright\ is used; first apply $\left(
b-1,b\right)  \left(  b-2,b-1\right)  \ldots\left(  a,a+1\right)  $ to $T_{0}%
$; this moves $b$ to the column with highest possible number; then repeat the
process with $\mathcal{S}_{\left[  a,b-1\right]  }$, or $\mathcal{S}_{\left[
a,b-k\right]  }$ if $b-k+1,\ldots,b$ are now in the same row, and so on). The
coefficient of $v_{T_{1}}$ in $f$ is $P_{0}\left(  T_{1};a,b\right)  $. Thus
$c=\frac{n_{0}}{P_{0}\left(  T_{1};a,b\right)  }$in $f$. Finally
\begin{align*}
\left\langle f,f\right\rangle _{0}  &  =\frac{1}{c}\left\langle f_{1}%
,f\right\rangle _{0}=\frac{1}{c}\sum_{w\in\mathcal{S}_{\left[  a,b\right]  }%
}\left\langle wv_{T_{0}},f\right\rangle _{0}\\
&  =\frac{\left(  b-a\right)  !}{c}\left\langle v_{T_{0}},f\right\rangle
_{0}=\frac{\left(  b-a\right)  !}{c}\left\langle v_{T_{0}},v_{T_{0}%
}\right\rangle .
\end{align*}
This completes the proof.
\end{proof}

It is straightforward to extend these methods to the case $H=\mathcal{S}%
_{\left[  a_{1},b_{1}\right]  }\times\mathcal{S}_{\left[  a_{2},b_{2}\right]
}\times\ldots\mathcal{S}_{\left[  a_{n},b_{n}\right]  }$ where $1\leq
a_{1}<b_{1}<a_{2}<b_{2}<\ldots<a_{n}<b_{n}\leq N$. This requires a tableau
$T_{0}\in Y\left(  \tau\right)  $ satisfying condition $\left[  a_{i}%
,b_{i}\right]  _{\mathrm{cm}}$ for $1\leq i\leq n$. Then%
\[
f=\sum_{T\in Y\left(  T_{0};H\right)  }\prod_{i=1}^{n}P_{0}\left(
T;a_{i},b_{i}\right)  v_{T}%
\]
is the unique $H$-invariant element of $V_{T_{0}}\left(  H\right)  $.

We turn to the problem of antisymmetric vectors in $V_{T}\left(  H\right)  $.
The previous arguments transfer almost directly by transposing tableaux and
inserting minus signs at appropriate places.

Say $T$ satisfies condition $\left[  a,b\right]  _{\mathrm{rw}}$ if $a\leq
i<j\leq m$ implies $\mathrm{rw}\left(  i,T\right)  \neq\mathrm{rw}\left(
j,T\right)  $ (the entries $a,a+1,\ldots,b$ are in distinct rows of $T$). Fix
some $T$ satisfying this condition and consider the subspace $V_{T}\left(
\mathcal{S}_{\left[  a,b\right]  }\right)  $. Let $T_{0}\in Y\left(
T;\mathcal{S}_{\left[  a,b\right]  }\right)  $ satisfy $\mathrm{cm}\left(
i,T\right)  \geq\mathrm{cm}\left(  j,T\right)  $ for $a\leq i<j\leq b$.

\begin{definition}
Suppose $T\in Y\left(  \tau\right)  $ satisfies condition $\left[  a,b\right]
_{\mathrm{cm}}$ then let%
\[
P_{1}\left(  T;a,b\right)  =\prod\limits_{a\leq i<j\leq b,~\mathrm{cm}\left(
i,T\right)  <\mathrm{cm}\left(  j,T\right)  }\frac{c\left(  j,T\right)
-c\left(  i,T\right)  }{1-c\left(  j,T\right)  +c\left(  i,T\right)  }.
\]

\end{definition}

As before we use the basic set $Y\left(  T_{0};\mathcal{S}_{\left[
a,b\right]  }\right)  $ to produce an anti-symmetric vector. Note
$P_{1}\left(  T_{0};a,b\right)  =1.$

\begin{proposition}
\label{avnorm}Let $f=\sum_{T\in Y\left(  T_{0};\mathcal{S}_{\left[
a,b\right]  }\right)  }P_{1}\left(  T;a,b\right)  v_{T}$ then $s_{i}f=-f$ for
$a\leq i<b$ and $wf=$\textrm{sgn}$\left(  w\right)  f$ for all $w\in
\mathcal{S}_{\left[  a,b\right]  }$. Let $n_{0}=\#\left\{  w\in\mathcal{S}%
_{\left[  a,b\right]  }:wv_{T_{0}}=\pm v_{T_{0}}\right\}  $ then $\left\Vert
f\right\Vert ^{2}=\frac{\left(  b-a\right)  !}{n_{0}}\left\vert P_{1}\left(
T_{1};a,b\right)  \right\vert ~\left\Vert v_{T_{0}}\right\Vert ^{2}$.
\end{proposition}

\begin{proof}
Suppose $a\leq i<b$ then let $A=\left\{  T\in Y\left(  T_{0};\mathcal{S}%
_{\left[  a,b\right]  }\right)  :\mathrm{cm}\left(  i,T\right)  =\mathrm{cm}%
\left(  i+1,T\right)  \right\}  $ and $B=\left\{  T\in Y\left(  T_{0}%
;\mathcal{S}_{\left[  a,b\right]  }\right)  :\mathrm{rw}\left(  i,T\right)
<\mathrm{rw}\left(  i+1,T\right)  \right\}  $ (also $\mathrm{cm}\left(
i,T\right)  >\mathrm{cm}\left(  i+1,T\right)  $ for $T\in B$); $T\in A$
implies $s_{i}v_{T}=-v_{T}$. Then%
\[
f=\sum_{T\in A}P_{1}\left(  T;a,b\right)  v_{T}+\sum_{T\in B}\left(
P_{1}\left(  T;a,b\right)  v_{T}+P_{1}\left(  s_{i}T;a,b\right)  v_{s_{i}%
T}\right)  .
\]
Fix $T\in B$ and compute $P_{1}\left(  T;a,b\right)  /P_{1}\left(
s_{i}T;a,b\right)  $ using Lemma \ref{gprod}; set $g_{mn}\left(  T\right)  =1$
if $\mathrm{cm}\left(  m,T\right)  \geq\mathrm{cm}\left(  n,T\right)  $ and
$g_{mn}\left(  T\right)  =\frac{c\left(  n,T\right)  -c\left(  m,T\right)
}{1-c\left(  n,T\right)  +c\left(  m,T\right)  }$ if $\mathrm{cm}\left(
m,T\right)  <\mathrm{cm}\left(  n,T\right)  $. Then $g_{i,i+1}\left(
T\right)  =1$ and $g_{i,i+1}\left(  s_{i}T\right)  =\frac{c\left(
i+1,s_{i}T\right)  -c\left(  i,s_{i}T\right)  }{1-c\left(  i+1,s_{i}T\right)
+c\left(  i,s_{i}T\right)  }=\frac{1}{-1-b_{i}\left(  s_{i}T\right)  }%
=\frac{1}{b_{i}\left(  T\right)  -1}$. Thus $P_{0}\left(  T;a,b\right)
/P_{0}\left(  s_{i}T;a,b\right)  =b_{i}\left(  T\right)  -1$ and $s_{i}f=-f$
by Proposition \ref{siv}. The norm formula follows from the proof of Corollary
\ref{svnorm} with some small modifications to take care of sign-changes.
\end{proof}

There are corresponding statements for $H=\mathcal{S}_{\left[  a_{1}%
,b_{1}\right]  }\times\mathcal{S}_{\left[  a_{2},b_{2}\right]  }\times
\ldots\mathcal{S}_{\left[  a_{n},b_{n}\right]  }$, using disjoint intervals.
The branching theorem for the restriction of irreducible representations of
$\mathcal{S}_{N}$ to those of the parabolic subgroups (like $H$) implicitly
appears in the previous discussion, in connection with the conditions $\left[
a,b\right]  _{cm}$ and $\left[  a,b\right]  _{rw}$.

\section{Dunkl operators}

Let $\kappa$ be a transcendental (formal parameter) and set $\mathbb{F}%
=\mathbb{Q}\left(  \kappa\right)  $. Consider the space $\mathcal{P}\otimes
V_{\tau}=\mathrm{span}_{\mathbb{F}}\left\{  x^{\alpha}v_{T}:\alpha
\in\mathbb{N}_{0}^{N},T\in Y\left(  \tau\right)  \right\}  $, polynomials
$p\left(  x\right)  $ on $%
%TCIMACRO{\U{211d} }%
%BeginExpansion
\mathbb{R}
%EndExpansion
^{N}\ $with values in $V_{\tau}$. The space is an $\mathcal{S}_{N}$-module
with the action $w\left(  x^{\alpha}v_{T}\right)  =x^{w\alpha}\left(
wv_{T}\right)  $ for $w\in\mathcal{S}_{N}$, extended to all of $\mathcal{P}%
\otimes V_{\tau}$ by linearity. For $p\in\mathcal{P}$ and $u\in V_{\tau}$ and
$1\leq i\leq N$ let%
\begin{equation}
\mathcal{D}_{i}\left(  p\left(  x\right)  u\right)  :=\frac{\partial}{\partial
x_{i}}p\left(  x\right)  u+\kappa\sum_{j=1,j\neq i}^{N}\frac{p\left(
x\right)  -p\left(  x\left(  i,j\right)  \right)  }{x_{i}-x_{j}}\left(
i,j\right)  u. \label{defdi}%
\end{equation}
The definition is extended to $\mathcal{P}\otimes V_{\tau}$ by linearity. Then
$\mathcal{D}_{i}\mathcal{D}_{j}=\mathcal{D}_{j}\mathcal{D}_{i}$ for $1\leq
i,j\leq N$. The proof is a straightforward adaptation of the original proof
for scalar polynomials $p\left(  x\right)  $ (see \cite[Ch. 4]{DX}). There are
important commutators (appearing in the definition of the rational Cherednik
algebra, the algebra generated by $\Bbbk\mathcal{S}_{N}$ and $\left\{
x_{i},\mathcal{D}_{i}:i\in\left[  1,N\right]  \right\}  $):%
\begin{align}
\mathcal{D}_{i}x_{j}-x_{j}\mathcal{D}_{i}  &  =-\kappa\left(  i,j\right)
,i\neq j\label{dxxd}\\
\mathcal{D}_{i}x_{i}-x_{i}\mathcal{D}_{i}  &  =1+\kappa\sum_{j\neq i}\left(
i,j\right)  .\nonumber
\end{align}

\begin{definition}
The space $\mathcal{P}\otimes V_{\tau}$ equipped with the action of
$\mathbb{F}\mathcal{S}_{N}$ and $\left\{  x_{i},\mathcal{D}_{i}:i\in\left[
1,N\right]  \right\}  $ is a standard module of the rational Cherednik algebra
and is denoted by $M\left(  \tau\right)  $.
\end{definition}

The representation theory of rational Cherednik algebras is described in the
survey \cite{R} by Rouquier. For $p\left(  x\right)  \in\mathcal{P}\otimes
V_{\tau}$ set
\[
\mathcal{U}_{i}p\left(  x\right)  =\mathcal{D}_{i}\left(  x_{i}p\left(
x\right)  \right)  -\kappa\sum_{j=1}^{i-1}\left(  i,j\right)  p\left(
x\right)  ,1\leq i\leq N.
\]
The operators $\mathcal{U}_{i}$ also commute pairwise. They have a
triangularity property (a special case of a result of Griffeth \cite{G} for
the complex reflection groups $G\left(  p,r,N\right)  $). There is an
important function on compositions:

\begin{definition}
For $\alpha\in\mathbb{N}_{0}^{N}$ and $1\leq i\leq N$ let $r\left(
\alpha,i\right)  :=\#\left\{  j:\alpha_{j}>\alpha_{i}\right\}  +\#\left\{
j:1\leq j\leq i,\alpha_{j}=\alpha_{i}\right\}  $ be the rank function.
\end{definition}

A consequence of the definition is that $r\left(  \alpha,i\right)  <r\left(
\alpha,j\right)  $ is equivalent to $\alpha_{i}>\alpha_{j}$, or $\alpha
_{i}=\alpha_{j}$ and $i<j$. For any $\alpha$ the function $i\mapsto r\left(
\alpha,i\right)  $ is one-to-one on $\left\{  1,2,\ldots,N\right\}  $. Let
$w_{\alpha}$ denote the inverse function, thus $r\left(  \alpha,w_{\alpha
}\left(  i\right)  \right)  =i$. Further $\alpha$ is a partition if and only
if $r\left(  \alpha,i\right)  =i$ for all $i$. In general $\left(  w_{\alpha
}^{-1}\alpha\right)  _{i}=\alpha_{w_{\alpha}\left(  i\right)  }$ for $1\leq
i\leq N$, and thus $w_{\alpha}^{-1}\alpha$ is a partition, denoted by
$\alpha^{+}$. The order on compositions is derived from the dominance order.

\begin{definition}
For $\alpha,\beta\in\mathbb{N}_{0}^{N}$ the partial order $\alpha\succ\beta$
($\alpha$ dominates $\beta$) means that $\alpha\neq\beta$ and $\sum_{i=1}%
^{j}\alpha_{i}\geq\sum_{i=1}^{j}\beta_{i}$ for $1\leq j\leq N$; and
$\alpha\vartriangleright\beta$ means that $\left\vert \alpha\right\vert
=\left\vert \beta\right\vert $ and either $\alpha^{+}\succ\beta^{+}$ or
$\alpha^{+}=\beta^{+}$ and $\alpha\succ\beta$.
\end{definition}

There are some results useful in analyzing $\mathcal{U}_{i}x^{\alpha}u$. Let
$\varepsilon\left(  i\right)  $ be the $i$th standard basis vector in $%
%TCIMACRO{\U{2115} }%
%BeginExpansion
\mathbb{N}
%EndExpansion
_{0}^{N}$, for $1\leq i\leq N$. By \cite[Lemma 8.2.3]{DX} the following hold
for $\alpha\in%
%TCIMACRO{\U{2115} }%
%BeginExpansion
\mathbb{N}
%EndExpansion
_{0}^{N}$:

\begin{enumerate}
\item if $\alpha_{i}>\alpha_{j}$ and $i<j$ then $\left(  i,j\right)
\alpha\vartriangleleft\alpha$;

\item $\alpha^{+}\trianglerighteq\alpha$;

\item if $1\leq s<\alpha_{i}-\alpha_{j}$ then $\alpha^{+}\vartriangleright
\left(  \alpha-s\left(  \varepsilon\left(  i\right)  -\varepsilon\left(
j\right)  \right)  \right)  ^{+}$.
\end{enumerate}

The following is a consequence of these relations and an easy computation.

\begin{lemma}
For $\alpha\in%
%TCIMACRO{\U{2115} }%
%BeginExpansion
\mathbb{N}
%EndExpansion
_{0}^{N}$ and $i\neq j$ let $B_{ij}x^{\alpha}=\left(  x_{i}x^{\alpha}%
-x_{j}\left(  i,j\right)  x^{\alpha}\right)  /\left(  x_{i}-x_{j}\right)  $,
then\newline(1) if $\alpha_{i}=\alpha_{j}$ then $B_{ij}x^{\alpha}=x^{\alpha}%
;$\newline(2) if $\alpha_{i}>\alpha_{j}$ then $B_{ij}x^{\alpha}=x^{\alpha
}+\left(  i,j\right)  x^{\alpha}+\sum\limits_{s=1}^{\alpha_{i}-\alpha_{j}%
-1}x^{\alpha-s\left(  \varepsilon\left(  i\right)  -\varepsilon\left(
j\right)  \right)  }$ and $\alpha^{+}\vartriangleright\left(  \alpha-s\left(
\varepsilon\left(  i\right)  -\varepsilon\left(  j\right)  \right)  \right)
^{+}$ for $1\leq s\leq\alpha_{i}-\alpha_{j}-1$;\newline(3) if $\alpha
_{i}<\alpha_{j}$ then $B_{ij}x^{\alpha}=-\sum\limits_{s=1}^{\alpha_{j}%
-\alpha_{i}-1}x^{\alpha-s\left(  \varepsilon\left(  j\right)  -\varepsilon
\left(  i\right)  \right)  }$ and $\alpha^{+}\vartriangleright\left(
\alpha-s\left(  \varepsilon\left(  j\right)  -\varepsilon\left(  i\right)
\right)  \right)  ^{+}$ for $1\leq s\leq\alpha_{j}-\alpha_{i}-1$.
\end{lemma}

The following Proposition can be elegantly stated in terms of conjugates of
Jucys-Murphy elements. Recall the conjugation relation $w\left(  i,j\right)
w^{-1}=\left(  w\left(  i\right)  ,w\left(  j\right)  \right)  $.

\begin{definition}
\label{w_a}For $\alpha\in%
%TCIMACRO{\U{2115} }%
%BeginExpansion
\mathbb{N}
%EndExpansion
_{0}^{N}$ and $1\leq i\leq N$ let $\omega_{i}^{\alpha}:=w_{\alpha}%
\omega_{r\left(  \alpha,i\right)  }w_{\alpha}^{-1}$, where $w_{\alpha}$ is the
inverse of $r\left(  \alpha,\cdot\right)  $. Equivalently $\omega_{i}^{\alpha
}=\sum\left\{  \left(  i,j\right)  :r\left(  \alpha,j\right)  >r\left(
\alpha,i\right)  \right\}  $.
\end{definition}

To justify the second equation observe that%
\[
w_{\alpha}\omega_{r\left(  \alpha,i\right)  }w_{\alpha}^{-1}=\sum
\limits_{r\left(  \alpha,i\right)  <j}\left(  w_{\alpha}\left(  r\left(
\alpha,i\right)  \right)  ,w_{\alpha}\left(  j\right)  \right)  =\sum
\limits_{r\left(  \alpha,i\right)  <j}\left(  i,w_{\alpha}\left(  j\right)
\right)
\]
and $r\left(  \alpha,w_{\alpha}\left(  j\right)  \right)  =j$.

\begin{proposition}
\label{uixg}Suppose $\alpha\in%
%TCIMACRO{\U{2115} }%
%BeginExpansion
\mathbb{N}
%EndExpansion
_{0}^{N}$, $u\in V_{\tau}$ and $1\leq i\leq N$ then%
\begin{equation}
\mathcal{U}_{i}x^{\alpha}u=x^{\alpha}\left[  \left(  \alpha_{i}+1\right)
u+\kappa\omega_{i}^{\alpha}u\right]  +\kappa\sum_{\beta\vartriangleleft\alpha
}x^{\beta}u_{\beta},\nonumber
\end{equation}
where each $u_{\beta}=0$ or $\pm\left(  i,j\right)  u$ for some $j$.
\end{proposition}

\begin{proof}
Let $q_{\alpha}$ denote elements of $\mathrm{span}\left\{  x^{\beta}%
:\beta\vartriangleleft\alpha\right\}  $. In the case $1\leq j<i$ the
coefficient of $\kappa\left(  i,j\right)  u$ is $B_{ij}x^{\alpha}-\left(
i,j\right)  x^{\alpha}$ which equals (1) $0$ if $\alpha_{i}=\alpha_{j}$, (2)
$x^{\alpha}+q_{\alpha}$ if $\alpha_{i}>\alpha_{j},$(3) $-\left(  \left(
i,j\right)  x^{\alpha}+q_{\alpha}\right)  $ if $\alpha_{j}>\alpha_{i}$, so
that $\left(  i,j\right)  \alpha\vartriangleleft\alpha$. In the case $i<j\leq
N$ the coefficient of $\kappa\left(  i,j\right)  u$ is $B_{ij}x^{\alpha}$
which equals (1) $x^{\alpha}$ if $\alpha_{i}=\alpha_{j}$, (2) $x^{\alpha
}+\left(  i,j\right)  x^{\alpha}+q_{\alpha}$ if $\alpha_{i}>\alpha_{j}$, so
that $\left(  i,j\right)  \alpha\vartriangleleft\alpha$, (3) $q_{\alpha}$ if
$\alpha_{i}<\alpha_{j}$. Thus $\kappa x^{\alpha}\left(  i,j\right)  u$ appears
in $\mathcal{U}_{i}x^{\alpha}u$ exactly when $\alpha_{i}>\alpha_{j}$ or
$\alpha_{i}=\alpha_{j}$ and $j>i$, that is, $r\left(  \alpha,j\right)
>r\left(  \alpha,i\right)  $.
\end{proof}

Following Griffeth we define an order on the pairs $\left\{  \left(
\alpha,u\right)  :\alpha\in%
%TCIMACRO{\U{2115} }%
%BeginExpansion
\mathbb{N}
%EndExpansion
_{0}^{N}\right\}  $: $\left(  \alpha,u_{1}\right)  \vartriangleright\left(
\beta,u_{2}\right)  $ means that $\alpha\vartriangleright\beta$. For this
order the leading term of $\mathcal{U}_{i}x^{\alpha}u$ is $x^{\alpha}\left(
\alpha_{i}+1+\kappa\omega_{i}^{\alpha}\right)  u$.

\section{Nonsymmetric Jack polynomials}

This section presents the structure of the simultaneous eigenvectors of
$\left\{  \mathcal{U}_{i}:1\leq i\leq N\right\}  $ in $M\left(  \tau\right)
$. These are vector-valued generalizations of the nonsymmetric Jack
polynomials (see \cite[Ch. 8]{DX}). The operators $\mathcal{U}_{i}$ are
self-adjoint with respect to the contravariant form, which is described as follows:

The contravariant form $\left\langle \cdot,\cdot\right\rangle $ on $M\left(
\tau\right)  $ is the canonical symmetric $\mathcal{S}_{N}$-invariant bilinear
form, extending the form $\left\langle \cdot,\cdot\right\rangle _{0}$ on
$V_{\tau}$, : such that%
\[
\left\langle x_{i}f,g\right\rangle =\left\langle f,\mathcal{D}_{i}%
g\right\rangle ,i\in\left[  1,N\right]  ,f,g\in M\left(  \tau\right)  .
\]
An existence proof can be based on the operator $\sum\limits_{i=1}^{N}%
x_{i}\mathcal{D}_{i}+\kappa\sum\limits_{1\leq i<j\leq N}\left(  i,j\right)  $
and induction. The important properties of the form are:

\begin{enumerate}
\item if $f\in\mathcal{P}_{m}\otimes V_{\tau},g\in\mathcal{P}_{n}\otimes
V_{\tau}$ and $m\neq n$ then $\left\langle f,g\right\rangle =0;$

\item if $w\in\mathcal{S}_{N}$ then $\left\langle wf,wg\right\rangle
=\left\langle f,g\right\rangle $ for all $f,g\in M\left(  \tau\right)  $, if
$1\leq i<j\leq N$ then $\left\langle \left(  i,j\right)  f,g\right\rangle
=\left\langle f,\left(  i,j\right)  g\right\rangle $;

\item if $1\leq i\leq N$ and $f,g\in M\left(  \tau\right)  $ then
$\left\langle \mathcal{D}_{i}x_{i}f,g\right\rangle =\left\langle
f,\mathcal{D}_{i}x_{i}g\right\rangle $.
\end{enumerate}

We use $\left\Vert f\right\Vert ^{2}$ to denote $\left\langle f,f\right\rangle
$ although the form may not be positive-definite . For a specific value
$\kappa\in\mathbb{Q}$ the kernel of the form, that is, $\left\{
f:\left\langle g,f\right\rangle =0,\forall g\in M\left(  \tau\right)
\right\}  $, is called the \textit{radical} of $M\left(  \tau\right)  $ and
denoted $J_{\kappa}\left(  \tau\right)  $, and the quotient module $M\left(
\tau\right)  /J_{\kappa}\left(  \tau\right)  $ is denoted $L_{\kappa}\left(
\tau\right)  $. Values of $\kappa$ such that $J_{\kappa}\left(  \tau\right)
\neq\left(  0\right)  $ are called \textit{singular} values.

We note that if $\lambda\in%
%TCIMACRO{\U{2115} }%
%BeginExpansion
\mathbb{N}
%EndExpansion
_{0}^{N,+}$ then the leading term in $\mathcal{U}_{i}x^{\lambda}u$ is
$x^{\lambda}\left(  \lambda_{i}+1+\kappa\omega_{i}\right)  u$; this suggests
that eigenvectors of $\omega_{i}$ have good properties under the action of
$\mathcal{U}_{i}$. For compositions the coordinates have to be appropriately
permuted. From (\ref{uixg}) we see that for $T\in Y\left(  \tau\right)  $ and
$\alpha\in%
%TCIMACRO{\U{2115} }%
%BeginExpansion
\mathbb{N}
%EndExpansion
_{0}^{N}$ the leading term in $\mathcal{U}_{i}x^{\alpha}w_{\alpha}v_{T}$ is
$\left(  \alpha_{i}+1+\kappa c\left(  r\left(  \alpha,i\right)  ,T\right)
\right)  x^{\alpha}w_{\alpha}v_{T}$, because $\omega_{i}^{\alpha}w_{\alpha
}v_{T}=w_{\alpha}\omega_{r\left(  \alpha,i\right)  }v_{T}=c\left(  r\left(
\alpha,i\right)  ,T\right)  w_{\alpha}v_{T}$. For any $n\in%
%TCIMACRO{\U{2115} }%
%BeginExpansion
\mathbb{N}
%EndExpansion
_{0}$ the set \newline$\left\{  x^{\alpha}w_{\alpha}v_{T}:\alpha\in%
%TCIMACRO{\U{2115} }%
%BeginExpansion
\mathbb{N}
%EndExpansion
_{0}^{N},\left\vert \alpha\right\vert =n,T\in Y\left(  \tau\right)  \right\}
$ is a basis of $M_{n}\left(  \tau\right)  :=\mathcal{P}_{n}\otimes V_{\tau}$
on which the operators $\mathcal{U}_{i}$ act in a triangular manner (with
respect to $\vartriangleright$). For $\alpha\in%
%TCIMACRO{\U{2115} }%
%BeginExpansion
\mathbb{N}
%EndExpansion
_{0}^{N},T\in Y\left(  \tau\right)  $, let
\[
\xi_{i}\left(  \alpha,T\right)  =\alpha_{i}+1+\kappa c\left(  r\left(
\alpha,i\right)  ,T\right)  ,1\leq i\leq N.
\]
For any $\beta\neq\alpha$ and $\left\vert \beta\right\vert =\left\vert
\alpha\right\vert $ there is at least one $i$ such that $\alpha_{i}\neq0$ and
$\alpha_{i}\neq\beta_{i}$ thus $\xi_{i}\left(  \alpha,T\right)  \neq\xi
_{i}\left(  \beta,T^{\prime}\right)  $ for any $T,T^{\prime}\in Y\left(
\tau\right)  $ (and generic $\kappa$). (The restriction to $\alpha_{i}\neq0$
is needed in the next section; if $\left\vert \alpha\right\vert =\left\vert
\beta\right\vert $ and $\alpha_{i}\neq0$ implies $\alpha_{i}=\beta_{i}$ then
$\alpha=\beta$.) Thus there exists a basis of simultaneous eigenvectors of
$\left\{  \mathcal{U}_{i}:1\leq i\leq N\right\}  $. The following is the
specialization to $\mathcal{S}_{N}$ of Griffeth's construction \cite[Theorem
5.2]{G} of nonsymmetric Jack polynomials.

\begin{proposition}
For $\alpha\in%
%TCIMACRO{\U{2115} }%
%BeginExpansion
\mathbb{N}
%EndExpansion
_{0}^{N},T\in Y\left(  \tau\right)  $ there exists a unique element
$\zeta_{\alpha,T}$ of $M\left(  \tau\right)  $ such that $\mathcal{U}_{i}%
\zeta_{\alpha,T}=\xi_{i}\left(  \alpha,T\right)  \zeta_{\alpha,T}$ for $1\leq
i\leq N$ and%
\[
\zeta_{\alpha,T}\left(  x\right)  =x^{\alpha}w_{\alpha}v_{T}+\sum
_{\beta\vartriangleleft\alpha}x^{\beta}g_{\beta\alpha},
\]
where $g_{\beta\alpha}\in V_{\tau}$.
\end{proposition}

The existence of this set of simultaneous eigenvectors of $\left\{
\mathcal{U}_{i}:1\leq i\leq N\right\}  $ follows from the triangular property,
the commutativity, and the separation properties of the eigenvalues $\left(
\alpha,T\right)  \mapsto\left[  \xi_{i}\left(  \alpha,T\right)  \right]
_{i=1}^{N}$.

Because each $\mathcal{U}_{i}$ is self-adjoint for $\left\langle \cdot
,\cdot\right\rangle $ we have $\left\langle \zeta_{\alpha,T},\zeta
_{\beta,T^{\prime}}\right\rangle =0$ when $\alpha\neq\beta$ or $T\neq
T^{\prime}$.

We consider the action of $\mathcal{S}_{N}$ on the polynomials $\zeta
_{\alpha,T}$. As usual there are explicit formulae for the action of
$s_{i}=\left(  i,i+1\right)  $ based on the commutations $\mathcal{U}_{j}%
s_{i}=s_{i}\mathcal{U}_{j}$ for $j\neq i,i+1$ and $s_{i}\mathcal{U}_{i}%
s_{i}=\mathcal{U}_{i+1}+\kappa$. These are special cases of \cite[Theorem
5.3]{G}, however we use the nonnormalized basis for $V_{\tau}$ rather than the
orthonormal one used there (so coefficients in $\mathbb{Q}\left(
\kappa\right)  $ suffice). As in \cite{G} let $\sigma_{i}$ denote the formal
operator $s_{i}+\frac{\kappa}{\mathcal{U}_{i+1}-\mathcal{U}_{i}}$; suppose
$f\in M\left(  \tau\right)  $ and $\mathcal{U}_{j}f=\lambda_{j}f$ for $1\leq
j\leq N$ (with $\lambda_{j}\in%
%TCIMACRO{\U{211a} }%
%BeginExpansion
\mathbb{Q}
%EndExpansion
\left(  \kappa\right)  $ and $\lambda_{i}\neq\lambda_{i+1}$) then
$\mathcal{U}_{j}\sigma_{i}f=\lambda_{j}\sigma_{i}f$ for $j\neq i,i+1$ and
$\mathcal{U}_{i}\sigma_{i}f=\lambda_{i+1}\sigma_{i}f$, $\mathcal{U}%
_{i+1}\sigma_{i}f=\lambda_{i}\sigma_{i}f$ (where $\sigma_{i}f=s_{i}%
f+\frac{\kappa}{\lambda_{i+1}-\lambda_{i}}f$. Specifically there are two main
cases $\alpha_{i}\neq\alpha_{i+1}$ and $\alpha_{i}=\alpha_{i+1}$. For
$\alpha\in%
%TCIMACRO{\U{2115} }%
%BeginExpansion
\mathbb{N}
%EndExpansion
_{0}^{N}$ and $T\in Y\left(  \tau\right)  $ let%
\begin{align*}
b_{i}\left(  \alpha,T\right)   &  =\frac{\kappa}{\xi_{i}\left(  \alpha
,T\right)  -\xi_{i+1}\left(  \alpha,T\right)  }\\
&  =\frac{\kappa}{\alpha_{i}-\alpha_{i+1}+\kappa\left(  c\left(  r\left(
\alpha,i\right)  ,T\right)  -c\left(  r\left(  \alpha,i+1\right)  ,T\right)
\right)  }.
\end{align*}

\begin{proposition}
\label{ai>ai1}Suppose $\alpha\in%
%TCIMACRO{\U{2115} }%
%BeginExpansion
\mathbb{N}
%EndExpansion
_{0}^{N}$ and $\alpha_{i}>\alpha_{i+1}$ for some $i<N$. Then
\begin{align*}
s_{i}\zeta_{\alpha,T}  &  =b_{i}\left(  \alpha,T\right)  \zeta_{\alpha
,T}+\left(  1-b_{i}\left(  \alpha,T\right)  ^{2}\right)  \zeta_{s_{i}\alpha
,T},\\
s_{i}\zeta_{s_{i}\alpha,T}  &  =\zeta_{\alpha,T}-b_{i}\left(  \alpha,T\right)
\zeta_{s_{i}\alpha,T};\\
\left\Vert \zeta_{\alpha,T}\right\Vert ^{2}  &  =\left(  1-b_{i}\left(
\alpha,T\right)  ^{2}\right)  \left\Vert \zeta_{s_{i}\alpha,T}\right\Vert
^{2}.
\end{align*}

\end{proposition}

\begin{proof}
The condition $\alpha_{i}\neq\alpha_{i+1}$ implies $r\left(  s_{i}%
\alpha,i\right)  =r\left(  \alpha,i+1\right)  $ and $r\left(  s_{i}%
\alpha,i+1\right)  =r\left(  \alpha,i\right)  $, thus $\xi_{i}\left(
s_{i}\alpha,T\right)  =\xi_{i+1}\left(  \alpha,T\right)  $ and $\xi
_{i+1}\left(  s_{i}\alpha,T\right)  =\xi_{i}\left(  \alpha,T\right)  $ (and
$\xi_{j}\left(  s_{i}\alpha,T\right)  =\xi_{j}\left(  \alpha,T\right)  $ for
$j\neq i,i+1$). Since the eigenvalues determine the eigenvectors uniquely we
have that
\begin{align*}
s_{i}\zeta_{\alpha,T}-b_{i}\left(  \alpha,T\right)  \zeta_{\alpha,T}  &
=a\zeta_{s_{i}\alpha,T},\\
s_{i}\zeta_{s_{i}\alpha,T}+b_{i}\left(  \alpha,T\right)  \zeta_{s_{i}%
\alpha,T}  &  =a^{\prime}\zeta_{\alpha,T},
\end{align*}
for some scalars $a,a^{\prime}$. The fact that $s_{i}^{2}=1$ implies
$aa^{\prime}=1-b_{i}\left(  \alpha,T\right)  ^{2}$. We show that $a^{\prime
}=1$ by finding the leading term in $s_{i}\zeta_{s_{i}\alpha,T},$namely
$x^{\alpha}s_{i}w_{s_{i}\alpha}v_{T}$. It remains to show that $w_{s_{i}%
\alpha}=s_{i}w_{\alpha}$, that is, $r\left(  s_{i}\alpha,s_{i}w_{\alpha
}\left(  j\right)  \right)  =j$ for all $j$. If $w_{\alpha}^{-1}\left(
j\right)  \neq i,i+1$ then $r\left(  s_{i}\alpha,s_{i}w_{\alpha}\left(
j\right)  \right)  =r\left(  \alpha,w_{\alpha}\left(  j\right)  \right)  =j$.
If $w_{\alpha}^{-1}\left(  j\right)  =i$ then $r\left(  s_{i}\alpha
,s_{i}w_{\alpha}\left(  j\right)  \right)  =r\left(  s_{i}\alpha,i+1\right)
=r\left(  \alpha,i\right)  =j$. The case $w_{\alpha}^{-1}\left(  j\right)
=i+1$ follows similarly. The second displayed equation shows that $\left\Vert
s_{i}\zeta_{s_{i}\alpha,T}\right\Vert ^{2}=\left\Vert \zeta_{s_{i}\alpha
,T}\right\Vert ^{2}=\left\Vert a^{\prime}\zeta_{\alpha,T}\right\Vert
^{2}-b_{i}\left(  \alpha,T\right)  ^{2}\left\Vert \zeta_{s_{i}\alpha
,T}\right\Vert ^{2}$.
\end{proof}

\begin{remark}
A necessary condition for the form $\left\langle \cdot,\cdot\right\rangle $ to
be positive-definite now becomes apparent: $b_{i}\left(  \alpha,T\right)
^{2}<1$ for all $i,\alpha,T$. The \textquotedblleft trivial\textquotedblright%
\ cases are $\tau=\left(  N\right)  $ and $\tau=\left(  1,\ldots,1\right)  $
for which $\kappa>-\frac{1}{N}$ and $\kappa<\frac{1}{N}$ are necessary and
sufficient, respectively. Otherwise let $h_{\tau}:=\tau_{1}+\ell\left(
\tau\right)  -1$, the maximum hook-length of $\tau$, then $-\frac{1}{h_{\tau}%
}<\kappa<\frac{1}{h_{\tau}}$ implies $b_{i}\left(  \alpha,T\right)  ^{2}<1$
for all $i,\alpha,T$. Note that $1\leq i,j\leq N,T\in Y\left(  \tau\right)  $
implies $\left\vert c\left(  i,T\right)  -c\left(  j,T\right)  \right\vert
\leq$ $h_{\tau}-1$.
\end{remark}

Etingof, Stoica and Griffeth \cite[Thm. 5.5]{ESG} found the complete
description of the set of values of $\kappa$ for which $L_{\kappa}\left(
\tau\right)  $ provides a unitary representation of the rational Cherednik
algebra. We can find an expression for $\left\Vert \zeta_{\alpha,T}\right\Vert
^{2}$ in terms of $\left\Vert \zeta_{\alpha^{+},T}\right\Vert ^{2}$, following
the approach used in \cite[Thm. 8.5.8]{DX}.

\begin{definition}
For $\alpha\in\mathbb{N}_{0}^{N},T\in Y\left(  \tau\right)  $ and
$\varepsilon=\pm$ let%
\[
\mathcal{E}_{\varepsilon}\left(  \alpha,T\right)  =\prod\limits_{1\leq i<j\leq
N,\alpha_{i}<\alpha_{j}}\left(  1+\frac{\varepsilon\kappa}{\alpha_{j}%
-\alpha_{i}+\kappa\left(  c\left(  r\left(  \alpha,j\right)  ,T\right)
-c\left(  r\left(  \alpha,i\right)  ,T\right)  \right)  }\right)  ,
\]
and let $\mathcal{E}_{2}\left(  \alpha,T\right)  =\mathcal{E}_{+}\left(
\alpha,T\right)  \mathcal{E}_{-}\left(  \alpha,T\right)  .$
\end{definition}

\begin{definition}
For $\alpha\in\mathbb{N}_{0}^{N}$ let $\mathrm{inv}\left(  \alpha\right)
=\#\left\{  \left(  i,j\right)  :1\leq i<j\leq N,\alpha_{i}<\alpha
_{j}\right\}  $.
\end{definition}

\begin{proposition}
Suppose $\alpha\in\mathbb{N}_{0}^{N},T\in Y\left(  \tau\right)  ,\varepsilon
=\pm$ and $\alpha_{i+1}>\alpha_{i}$ for some $i\in\left[  1,N-1\right]  $ then
$\mathcal{E}_{\varepsilon}\left(  s_{i}\alpha,T\right)  /\mathcal{E}%
_{\varepsilon}\left(  \alpha,T\right)  =1+\varepsilon b_{i}\left(
\alpha,T\right)  $.
\end{proposition}

\begin{proof}
Using an argument similar to that of Lemma \ref{gprod} we have $\mathcal{E}%
_{\varepsilon}\left(  s_{i}\alpha,T\right)  /\mathcal{E}_{\varepsilon}\left(
\alpha,T\right)  =1+\frac{\varepsilon\kappa}{\left(  s_{i}\alpha\right)
_{i+1}-\left(  s_{i}\alpha\right)  _{i}+\kappa\left(  c\left(  r\left(
s_{i}\alpha,i+1\right)  ,T\right)  -c\left(  r\left(  s_{i}\alpha,i\right)
,T\right)  \right)  }=1+\varepsilon b_{i}\left(  \alpha,T\right)  $, because
$r\left(  s_{i}\alpha,i+1\right)  =r\left(  \alpha,i\right)  $ and $r\left(
s_{i}\alpha,i\right)  =r\left(  \alpha,i+1\right)  .$
\end{proof}

\begin{corollary}
Suppose $\alpha\in\mathbb{N}_{0}^{N},T\in Y\left(  \tau\right)  $ then
$\left\Vert \zeta_{\alpha,T}\right\Vert ^{2}=\mathcal{E}_{2}\left(
\alpha,T\right)  ^{-1}\left\Vert \zeta_{\alpha^{+},T}\right\Vert ^{2}$.
\end{corollary}

\begin{proof}
Argue by induction on $\mathrm{inv}\left(  \alpha\right)  $. If the formula is
valid for some $\alpha$ with $\alpha_{i}>\alpha_{i+1}$ then by Proposition
\ref{ai>ai1}%
\begin{align*}
\left\Vert \zeta_{s_{i}\alpha,T}\right\Vert ^{2}  &  =\left(  1-b_{i}\left(
\alpha,T\right)  ^{2}\right)  ^{-1}\left\Vert \zeta_{\alpha,T}\right\Vert
^{2}\\
&  =\left(  1-b_{i}\left(  \alpha,T\right)  ^{2}\right)  ^{-1}\mathcal{E}%
_{2}\left(  \alpha,T\right)  ^{-1}\left\Vert \zeta_{\alpha^{+},T}\right\Vert
^{2}\\
&  =\mathcal{E}_{2}\left(  s_{i}\alpha,T\right)  ^{-1}\left\Vert \zeta
_{\alpha^{+},T}\right\Vert ^{2}.
\end{align*}
This completes the induction.
\end{proof}

Consider the case $\alpha_{i}=\alpha_{i+1}$ and let $I=r\left(  \alpha
,i\right)  $ so that $r\left(  \alpha,i+1\right)  =I+1$ and $b_{i}\left(
\alpha,T\right)  =\left(  c\left(  I,T\right)  -c\left(  I+1,T\right)
\right)  ^{-1}=b_{I}\left(  T\right)  $ (see Proposition \ref{siv}).
Furthermore $s_{i}w_{\alpha}=w_{\alpha}\left(  w_{\alpha}^{-1}\left(
i\right)  ,w_{\alpha}^{-1}\left(  i+1\right)  \right)  =w_{\alpha}\left(
I,I+1\right)  =w_{\alpha}s_{I}$. The transformation properties depend on the
positions of $I$ and $I+1$ in $T$.

\begin{proposition}
\label{ai=ai1}Suppose $\alpha\in%
%TCIMACRO{\U{2115} }%
%BeginExpansion
\mathbb{N}
%EndExpansion
_{0}^{N},T\in Y\left(  \tau\right)  $ and $\alpha_{i}=\alpha_{i+1}$ for some
$i<N$. For $I=r\left(  \alpha,T\right)  $ the following hold:\newline1) if
$b_{I}\left(  T\right)  =1$ then $s_{i}\zeta_{\alpha,T}=\zeta_{\alpha,T}%
$,\newline2) if $b_{I}\left(  T\right)  =-1$ then $s_{i}\zeta_{\alpha
,T}=-\zeta_{\alpha,T}$,\newline3) if $-\frac{1}{2}\leq b_{I}\left(  T\right)
<0$ then $s_{i}\zeta_{\alpha,T}=b_{I}\left(  T\right)  \zeta_{\alpha
,T}+\left(  1-b_{I}\left(  T\right)  ^{2}\right)  \zeta_{\alpha,s_{I}T}%
$,\newline4) if $0<b_{I}\left(  T\right)  \leq\frac{1}{2}$ then $s_{i}%
\zeta_{\alpha,T}=b_{I}\left(  T\right)  \zeta_{\alpha,T}+\zeta_{\alpha,s_{I}%
T}$.
\end{proposition}

\begin{proof}
It suffices to consider the action of $s_{i}$ on the leading term of
$\zeta_{\alpha,T}$. Indeed $s_{i}x^{\alpha}w_{\alpha}v_{T}=x^{\alpha}%
w_{\alpha}\left(  s_{I}v_{T}\right)  $ and we use the equations from
Proposition \ref{siv}.
\end{proof}

Note that in case (3) $\left\Vert \zeta_{\alpha,T}\right\Vert ^{2}=\left(
1-b_{I}\left(  T\right)  ^{2}\right)  \left\Vert \zeta_{\alpha,s_{I}%
T}\right\Vert ^{2}$ (and the reciprocal in case (4)). There is a raising
operator involving a cyclic shift and multiplication by $x_{N}$. From the
commutators \ref{dxxd} we obtain:%
\begin{align*}
\mathcal{U}_{i}x_{N}f  &  =x_{N}\left(  \mathcal{U}_{i}-\kappa\left(
i,N\right)  \right)  f,~1\leq i<N,\\
\mathcal{U}_{N}x_{N}f  &  =x_{N}\left(  1+\mathcal{D}_{N}x_{N}\right)  f.
\end{align*}
Let $\theta_{N}=s_{1}s_{2}\ldots s_{N-1}$ thus $\theta_{N}\left(  N\right)
=1$ and $\theta_{N}\left(  i\right)  =i+1$ for $1\leq i<N$ (a cyclic shift).
Then%
\begin{align*}
\mathcal{U}_{i}x_{N}f  &  =x_{N}\left(  \theta_{N}^{-1}\mathcal{U}_{i+1}%
\theta_{N}\right)  f,~1\leq i<N,\\
\mathcal{U}_{N}x_{N}f  &  =x_{N}\left(  1+\theta_{N}^{-1}\mathcal{U}_{1}%
\theta_{N}\right)  f.
\end{align*}
If $f$ satisfies $\mathcal{U}_{i}f=\lambda_{i}f$ for $1\leq i\leq N$ then
$\mathcal{U}_{i}\left(  x_{N}\theta_{N}^{-1}f\right)  =\lambda_{i+1}\left(
x_{N}\theta_{N}^{-1}f\right)  $ for $1\leq i<N$ and $\mathcal{U}_{N}\left(
x_{N}\theta_{N}^{-1}f\right)  =\left(  \lambda_{1}+1\right)  \left(
x_{N}\theta_{N}^{-1}f\right)  $. For $\alpha\in%
%TCIMACRO{\U{2115} }%
%BeginExpansion
\mathbb{N}
%EndExpansion
_{0}^{N}$ let $\phi\left(  \alpha\right)  :=\left(  \alpha_{2},\alpha
_{3},\ldots,\alpha_{N},\alpha_{1}+1\right)  $, then $x_{N}\theta_{N}%
^{-1}x^{\alpha}=x^{\phi\left(  \alpha\right)  }$.

\begin{proposition}
\label{xnz}Suppose $\alpha\in%
%TCIMACRO{\U{2115} }%
%BeginExpansion
\mathbb{N}
%EndExpansion
_{0}^{N},T\in Y\left(  \tau\right)  $, then $\zeta_{\phi\left(  \alpha\right)
,T}=x_{N}\theta_{N}^{-1}\zeta_{\alpha,T}$.
\end{proposition}

\begin{proof}
By straightforward arguments it follows that $r\left(  \phi\left(
\alpha\right)  ,i\right)  =r\left(  \alpha,i+1\right)  $ for $1\leq i<N$ and
$r\left(  \phi\left(  \alpha\right)  ,N\right)  =r\left(  \alpha,1\right)  $,
that is, $r\left(  \phi\left(  \alpha\right)  ,i\right)  =r\left(
\alpha,\theta_{N}\left(  i\right)  \right)  $ for all $i$. This is equivalent
to $r\left(  \phi\left(  \alpha\right)  ,\theta_{N}^{-1}\left(  w_{\alpha
}\left(  j\right)  \right)  \right)  =r\left(  \alpha,w_{\alpha}\left(
j\right)  \right)  =j$ for all $j$, or $w_{\phi\left(  \alpha\right)  }%
=\theta_{N}^{-1}w_{\alpha}$. The leading term $x^{\alpha}w_{\alpha}v_{T}$ of
$\zeta_{\alpha,T}$ is mapped to $x^{\phi\left(  \alpha\right)  }w_{\phi\left(
\alpha\right)  }v_{T}$ by $f\mapsto x_{N}\theta_{N}^{-1}f$. Note that
$\mathcal{U}_{i}\zeta_{\phi\left(  \alpha\right)  ,T}=\left(  \alpha
_{i+1}+1+\kappa c\left(  r\left(  \phi\left(  \alpha\right)  ,i\right)
,T\right)  \right)  \zeta_{\phi\left(  \alpha\right)  ,T}$ for $1\leq i<N$ and
$\mathcal{U}_{N}\zeta_{\phi\left(  \alpha\right)  ,T}=\left(  \alpha
_{1}+2+\kappa c\left(  r\left(  \phi\left(  \alpha\right)  ,N\right)
,T\right)  \right)  \zeta_{\phi\left(  \alpha\right)  ,T}$. Thus $x_{N}%
\theta_{N}^{-1}\zeta_{\alpha,T}$ and $\zeta_{\phi\left(  \alpha\right)  ,T}$
have the same eigenvalues for $\left\{  \mathcal{U}_{i}\right\}  $ and the
same coefficient of $x^{\phi\left(  \alpha\right)  }.$Hence $x_{N}\theta
_{N}^{-1}\zeta_{\alpha,T}=\zeta_{\phi\left(  \alpha\right)  ,T}$.
\end{proof}

\begin{corollary}
$\left\Vert \zeta_{\phi\left(  \alpha\right)  ,T}\right\Vert ^{2}=\left(
\alpha_{1}+1+\kappa c\left(  r\left(  \alpha,1\right)  ,T\right)  \right)
\left\Vert \zeta_{\alpha,T}\right\Vert ^{2}$.
\end{corollary}

\begin{proof}
Indeed $\left\Vert \zeta_{\phi\left(  \alpha\right)  ,T}\right\Vert
^{2}=\left\langle \theta_{N}^{-1}\zeta_{\alpha,T},\mathcal{D}_{N}x_{N}%
\theta_{N}^{-1}\zeta_{\alpha,T}\right\rangle =\left\langle \theta_{N}%
^{-1}\zeta_{\alpha,T},\theta_{N}^{-1}\mathcal{D}_{1}x_{1}\zeta_{\alpha
,T}\right\rangle \allowbreak=\left\langle \zeta_{\alpha,T},\mathcal{U}%
_{1}\zeta_{\alpha,T}\right\rangle =\xi_{1}\left(  \alpha,T\right)  \left\Vert
\zeta_{\alpha,T}\right\Vert ^{2}$.
\end{proof}

Griffeth \cite[Thm. 6.1]{G} showed the following:

\begin{theorem}
\label{normz}Suppose $\lambda\in%
%TCIMACRO{\U{2115} }%
%BeginExpansion
\mathbb{N}
%EndExpansion
_{0}^{N,+}$ and $T\in Y\left(  \tau\right)  $ then%
\[
\left\Vert \zeta_{\lambda,T}\right\Vert ^{2}=\left\Vert v_{T}\right\Vert
_{0}^{2}\prod\limits_{i=1}^{N}\left(  1+\kappa c\left(  i,T\right)  \right)
_{\lambda_{i}}\prod\limits_{1\leq i<j\leq N}\prod\limits_{l=1}^{\lambda
_{i}-\lambda_{j}}\left(  1-\frac{\kappa^{2}}{\left(  l+\kappa\left(  c\left(
i,T\right)  -c\left(  j,T\right)  \right)  \right)  ^{2}}\right)  .
\]

\end{theorem}

\begin{proof}
Argue by induction. Suppose $\lambda_{1}=\lambda_{2}=\ldots=\lambda
_{m}>\lambda_{m+1}$. Let%
\begin{align*}
\beta &  =\left(  \lambda_{1},\ldots,\lambda_{m-1},\lambda_{m+1}%
,\ldots,\lambda_{N},\lambda_{1}\right)  ,\\
\alpha &  =\left(  \lambda_{1}-1,\lambda_{1},\ldots,\lambda_{m-1}%
,\lambda_{m+1},\ldots,\lambda_{N}\right)  ,\\
\mu &  =\left(  \lambda_{1},\ldots,\lambda_{m-1},\lambda_{1}-1,\lambda
_{m+1},\ldots,\lambda_{N}\right)  .
\end{align*}
Thus $\beta=\phi\left(  \alpha\right)  $ and%
\begin{align*}
\left\Vert \zeta_{\beta,T}\right\Vert ^{2}  &  =\left(  \lambda_{1}+\kappa
c\left(  m,T\right)  \right)  \left\Vert \zeta_{\alpha,T}\right\Vert ^{2}\\
&  =\left(  \lambda_{1}+\kappa c\left(  m,T\right)  \right)  \mathcal{E}%
_{2}\left(  \alpha,T\right)  ^{-1}\left\Vert \zeta_{\mu,T}\right\Vert ^{2},\\
\left\Vert \zeta_{\lambda,T}\right\Vert ^{2}  &  =\mathcal{E}_{2}\left(
\beta,T\right)  \left\Vert \zeta_{\beta,T}\right\Vert ^{2}.
\end{align*}
We have%
\begin{align*}
\mathcal{E}_{\varepsilon}\left(  \alpha,T\right)   &  =\prod\limits_{j=2}%
^{m}\left(  1+\frac{\varepsilon\kappa}{1+\kappa\left(  c\left(  j-1,T\right)
-c\left(  m,T\right)  \right)  }\right)  ,\\
\mathcal{E}_{\varepsilon}\left(  \beta,T\right)   &  =\prod\limits_{j=m+1}%
^{N}\left(  1+\frac{\varepsilon\kappa}{\lambda_{1}-\lambda_{j}+\kappa\left(
c\left(  m,T\right)  -c\left(  j,T\right)  \right)  }\right)  .
\end{align*}
The validity of the formula for $\left\Vert \zeta_{\mu,T}\right\Vert ^{2}$
thus implies the validity for $\left\Vert \zeta_{\lambda,T}\right\Vert ^{2}$
(that is, the value of $\left\Vert \zeta_{\lambda,T}\right\Vert ^{2}%
/\left\Vert \zeta_{\mu,T}\right\Vert ^{2}$ from the formula agrees with
$\left(  \lambda_{1}+\kappa c\left(  m,T\right)  \right)  \dfrac
{\mathcal{E}_{2}\left(  \beta,T\right)  }{\mathcal{E}_{2}\left(
\alpha,T\right)  }$).
\end{proof}

\section{Symmetric and Antisymmetric Polynomials}

We consider symmetric and antisymmetric linear combinations of $\left\{
\zeta_{\alpha,T}\right\}  $. Recall
\begin{align*}
b_{i}\left(  \alpha,T\right)   &  =\frac{\kappa}{\alpha_{i}-\alpha
_{i+1}+\kappa\left(  c\left(  r\left(  \alpha,i\right)  ,T\right)  -c\left(
r\left(  \alpha,i+1\right)  ,T\right)  \right)  },\\
b_{i}\left(  T\right)   &  =\frac{1}{c\left(  i,T\right)  -c\left(
i+1,T\right)  },
\end{align*}
for $\alpha\in\mathbb{N}_{0}^{N},T\in Y\left(  \tau\right)  ,i\in\left[
1,N-1\right]  $. Here is a description of $s_{i}$-invariant polynomials for a
given $i$:

\begin{enumerate}
\item $\zeta_{\alpha,T}+\left(  1-b_{i}\left(  \alpha,T\right)  \right)
\zeta_{s_{i}\alpha,T}$, for $\alpha_{i}>\alpha_{i+1}$;

\item $\left(  b_{I}\left(  T\right)  +1\right)  \zeta_{\alpha,T}%
+\zeta_{\alpha,s_{I}T}$, for $\alpha_{i}=\alpha_{i+1},I=r\left(
\alpha,i\right)  $ and $0<b_{I}\left(  T\right)  \leq\frac{1}{2}$;

\item $\zeta_{\alpha,T}$, for $\alpha_{i}=\alpha_{i+1},I=r\left(
\alpha,i\right)  $ and $b_{I}\left(  T\right)  =1$ ($\mathrm{rw}\left(
I,T\right)  =\mathrm{rw}\left(  I+1,T\right)  $).
\end{enumerate}

The antisymmetric polynomials for $s_{i}$ ($s_{i}f=-f$) are

\begin{enumerate}
\item $\zeta_{\alpha,T}-\left(  1+b_{i}\left(  \alpha,T\right)  \right)
\zeta_{s_{i}\alpha,T}$, for $\alpha_{i}>\alpha_{i+1}$;

\item $\left(  b_{I}\left(  T\right)  -1\right)  \zeta_{\alpha,T}%
+\zeta_{\alpha,s_{I}T}$, for $\alpha_{i}=\alpha_{i+1},I=r\left(
\alpha,i\right)  $ and $0<b_{I}\left(  T\right)  \leq\frac{1}{2}$;

\item $\zeta_{\alpha,T}$, for $\alpha_{i}=\alpha_{i+1},I=r\left(
\alpha,i\right)  $ and $b_{I}\left(  T\right)  =-1$ ($\mathrm{cm}\left(
I,T\right)  =\mathrm{cm}\left(  I+1,T\right)  $).
\end{enumerate}

Now we construct invariants. In any orbit $\mathrm{span}\left\{
w\zeta_{\alpha,T}:w\in\mathcal{S}_{N}\right\}  $ there must be a polynomial
with leading term $x^{\alpha^{+}}$ so it suffices to consider the situation
$\zeta_{\lambda,T}$ for partitions $\lambda$. We collect concepts for use in
the sequel.

\begin{notation}
For $\lambda\in\mathbb{N}_{0}^{N,+}$ let $W_{\lambda}=\left\{  w\in
\mathcal{S}_{N}:w\lambda=\lambda\right\}  $, the stabilizer subgroup of
$\lambda$. Thus
\[
W_{\lambda}=\mathcal{S}_{\left[  a_{1},b_{1}\right]  }\times\mathcal{S}%
_{\left[  a_{2},b_{2}\right]  }\times\ldots\mathcal{S}_{\left[  a_{n}%
,b_{n}\right]  },
\]
where $1\leq a_{1}<b_{1}<a_{2}<b_{2}<\ldots<a_{n}<b_{n}\leq N$ (this means
$\lambda_{a_{1}}=\lambda_{b_{1}}>\lambda_{b_{1}+1}$ and so forth). These
intervals depend on $\lambda$ but we will not incorporate this into the
notation. Let $\lambda^{R}=\left(  \lambda_{N},\lambda_{N-1},\ldots
,\lambda_{1}\right)  \in\mathbb{N}_{0}^{N}$, the reverse of $\lambda$. The
permutation $w_{\lambda^{R}}$ is defined by $\left(  w_{\lambda^{R}}\right)
^{-1}\left(  i\right)  =r\left(  \lambda^{R},i\right)  ,i\in\left[
1,N\right]  $ (Definition \ref{w_a}).
\end{notation}

Generally $w_{\lambda^{R}}\neq w_{0}$ where $w_{0}$ is the longest permutation
given by $w_{0}\left(  i\right)  =N+1-i$ (example: $\lambda=\left(
3,2,2,1\right)  $ then $\left[  w_{\lambda^{R}}\left(  i\right)  \right]
_{i=1}^{4}=\left[  4,2,3,1\right]  $). The composition $\lambda^{R}$ is the
unique minimum for the order \textquotedblleft$\succ$\textquotedblright\ on
$\left\{  \alpha:\alpha^{+}=\lambda\right\}  $. For $\alpha^{+}=\lambda$ and
$T\in Y\left(  \tau\right)  $ the leading term of $\zeta_{\alpha,T}$ is
$x^{\alpha}w_{\alpha}v_{T}$ (where $w_{\alpha}^{-1}\left(  i\right)  =r\left(
\alpha,i\right)  $) and the minimality of $\lambda^{R}$ implies that the
expansion of $\zeta_{\lambda^{R},T}$ has no term of the form $x^{\alpha}u$
with $u\in V_{\tau}$ when $\alpha\neq\lambda^{R},\alpha^{+}=\lambda$. From the
expressions in (2) above we see that the subgroup $W_{\lambda}$ is an
important part of the analysis. The formulae developed in Section 2 will be used.

Consider first the simplest case where $v_{T}$ is $W_{\lambda}$-invariant,
that is, each interval $\left[  a_{i},b_{i}\right]  $ is contained in a row of
$T$, for $1\leq i\leq n$ ($\mathrm{rw}\left(  j,T\right)  =\mathrm{rw}\left(
b_{i},T\right)  $ for $a_{i}\leq j\leq b_{i}$). Then $\sum_{w\in
\mathcal{S}_{N}}w\zeta_{\lambda,T}=\sum_{\alpha^{+}=\lambda}A_{\alpha}%
\zeta_{\alpha,T}$ with coefficients to be determined.

\begin{theorem}
\label{symz}Suppose $\lambda\in\mathbb{N}_{0}^{N,+}$ and $T\in Y\left(
\tau\right)  $ such that $w\in W_{\lambda}$ implies $wv_{T}=v_{T}$ then the
polynomial $f_{\lambda,T}$ defined by%
\[
f_{\lambda,T}=\sum_{\alpha^{+}=\lambda}\mathcal{E}_{-}\left(  \alpha,T\right)
\zeta_{\alpha,T},
\]
is $\mathcal{S}_{N}$-invariant and
\[
\left\Vert f_{\lambda,T}^{s}\right\Vert ^{2}=\frac{N!}{\#W_{\lambda}}\frac
{1}{\mathcal{E}_{+}\left(  \lambda^{R},T\right)  }\left\Vert \zeta_{\lambda
,T}\right\Vert ^{2}.
\]

\end{theorem}

\begin{proof}
Fix $i\in\left[  1,N-1\right]  $ and let%
\begin{align*}
A  &  =\left\{  \alpha:\alpha^{+}=\lambda,\alpha_{i}=\alpha_{i+1}\right\}  ,\\
B  &  =\left\{  \alpha:\alpha^{+}=\lambda,\alpha_{i}>\alpha_{i+1}\right\}  .
\end{align*}
Write%
\[
f_{\lambda,T}=\sum_{\alpha\in A}\mathcal{E}_{-}\left(  \alpha,T\right)
\zeta_{\alpha,T}+\sum_{\alpha\in B}\left(  \mathcal{E}_{-}\left(
\alpha,T\right)  \zeta_{\alpha,T}+\mathcal{E}_{-}\left(  s_{i}\alpha,T\right)
\zeta_{s_{i}\alpha,T}\right)  .
\]
Suppose $\alpha\in A$ then $r\left(  i+1,\alpha\right)  =r\left(
i,\alpha\right)  +1$ thus the values $r\left(  i,\alpha\right)  $ and
$r\left(  i+1,\alpha\right)  $ belong to some interval $\left[  a_{j}%
,b_{j}\right]  $ (where $\mathcal{S}_{\left[  a_{j},b_{j}\right]  }$ is a
factor of $W_{\lambda}$) and are adjacent entries in some row of $T$, hence
$s_{i}\zeta_{\alpha,T}=\zeta_{\alpha,T}$. Next let $\alpha\in B$ then the
corresponding term in the sum is $\mathcal{E}_{-}\left(  \alpha,T\right)
\left(  \zeta_{\alpha,T}+\frac{\mathcal{E}_{-}\left(  s_{i}\alpha,T\right)
}{\mathcal{E}_{-}\left(  \alpha,T\right)  }\zeta_{s_{i}\alpha,T}\right)  $.
Using the techniques of Lemma \ref{gprod} we find that%
\begin{align*}
\dfrac{\mathcal{E}_{-}\left(  s_{i}\alpha,T\right)  }{\mathcal{E}_{-}\left(
\alpha,T\right)  }  &  =1-\frac{\kappa}{\left(  s_{i}\alpha\right)
_{i+1}-\left(  s_{i}\alpha\right)  _{i}+\kappa\left(  c\left(  r\left(
s_{i}\alpha,i+1\right)  ,T\right)  -c\left(  r\left(  s_{i}\alpha,i\right)
,T\right)  \right)  }\\
&  =1-\frac{\kappa}{\alpha_{i}-\alpha_{i+1}+\kappa\left(  c\left(  r\left(
\alpha,i\right)  ,T\right)  -c\left(  r\left(  \alpha,i\right)  ,T\right)
\right)  }=1-b_{i}\left(  \alpha,T\right)  ,
\end{align*}
and thus the term for $\alpha$ in the sum over $B$ is $s_{i}$-invariant.
Consider $g=\sum_{w\in\mathcal{S}_{N}}w\zeta_{\lambda^{R},T}$; since $g$ is
$\mathcal{S}_{N}$-invariant it must equal a constant multiple $\gamma$ of
$f_{\lambda,T}$. To find $\gamma$ consider the coefficients of $x^{\lambda
}v_{T}$ in $f_{\lambda,T}$ and $g$. The leading term of $\zeta_{\lambda^{R}%
,T}$ is $x^{\lambda^{R}}w_{\lambda^{R}}\left(  v_{T}\right)  .$ The
coefficient in $f_{\lambda,T}$ is $1$ (by definition of $\zeta_{\lambda,T}$).
The term $x^{\lambda}v_{T}$ appears in $w\zeta_{\lambda^{R},T}$ with
coefficient $1$ exactly when $w=w_{1}w_{\lambda^{R}}^{-1}$ for $w_{1}\in
W_{\lambda}$. Thus $g=\left(  \#W_{\lambda}\right)  f_{\lambda,T}$ and%
\begin{align*}
\left\Vert f_{\lambda,T}^{s}\right\Vert ^{2}  &  =\frac{1}{\#W_{\lambda}%
}\left\langle g,f_{\lambda,T}\right\rangle =\frac{1}{\#W_{\lambda}}\sum
_{w\in\mathcal{S}_{N}}\left\langle w\zeta_{\lambda^{R},T},f_{\lambda,T}%
^{s}\right\rangle \\
&  =\frac{N!}{\#W_{\lambda}}\left\langle \zeta_{\lambda^{R},T},f_{\lambda
,T}^{s}\right\rangle =\frac{N!}{\#W_{\lambda}}\mathcal{E}_{-}\left(
\lambda^{R},T\right)  \left\Vert \zeta_{\lambda^{R},T}\right\Vert ^{2}\\
&  =\frac{N!\mathcal{E}_{-}\left(  \lambda^{R},T\right)  }{\left(
\#W_{\lambda}\right)  \mathcal{E}_{2}\left(  \lambda^{R},T\right)  }\left\Vert
\zeta_{\lambda,T}\right\Vert ^{2}.
\end{align*}
This completes the proof.
\end{proof}

We turn to the corresponding antisymmetric function involving $\zeta
_{\lambda,T}$ where $v_{T}$ is antisymmetric for $W_{\lambda}$. That is each
interval $\left[  a_{i},b_{i}\right]  $ (appearing in $W_{\lambda}$) is
contained in a column of $T$, $a_{i}\leq j\leq b_{i}$ implies $\mathrm{cm}%
\left(  j,T\right)  =\mathrm{cm}\left(  b_{i},T\right)  $). The idea of the
sign of a permutation needs to be adapted to allow for $\lambda$ having some
values equal to each other; the number of inversions $\mathrm{inv}\left(
\alpha\right)  $ fills this role.

\begin{theorem}
\label{skewz}Suppose $\lambda\in\mathbb{N}_{0}^{N,+}$ and $T\in Y\left(
\tau\right)  $ such that $s_{i}\in W_{\lambda}$ implies $s_{i}v_{T}=-v_{T}$
then the polynomial $f_{\lambda,T}^{a}$ defined by%
\[
f_{\lambda,T}^{a}=\sum_{\alpha^{+}=\lambda}\left(  -1\right)  ^{\mathrm{inv}%
\left(  \alpha\right)  }\mathcal{E}_{+}\left(  \alpha,T\right)  \zeta
_{\alpha,T},
\]
is $\mathcal{S}_{N}$-alternating, and%
\[
\left\Vert f_{\lambda,T}^{a}\right\Vert ^{2}=\frac{N!}{\#W_{\lambda}}\frac
{1}{\mathcal{E}_{-}\left(  \lambda^{R},T\right)  }\left\Vert \zeta_{\lambda
,T}\right\Vert ^{2}.
\]

\end{theorem}

\begin{proof}
Fix $i\in\left[  1,N-1\right]  $ and let%
\begin{align*}
A  &  =\left\{  \alpha:\alpha^{+}=\lambda,\alpha_{i}=\alpha_{i+1}\right\}  ,\\
B  &  =\left\{  \alpha:\alpha^{+}=\lambda,\alpha_{i}>\alpha_{i+1}\right\}  .
\end{align*}
Note $\alpha\in B$ implies $\mathrm{inv}\left(  s_{i}\alpha\right)
=\mathrm{inv}\left(  \alpha\right)  +1$. Write%
\[
f_{\lambda,T}^{a}=\sum_{\alpha\in A}\left(  -1\right)  ^{\mathrm{inv}\left(
\alpha\right)  }\mathcal{E}_{+}\left(  \alpha,T\right)  \zeta_{\alpha,T}%
+\sum_{\alpha\in B}\left(  -1\right)  ^{\mathrm{inv}\left(  \alpha\right)
}\left(  \mathcal{E}_{+}\left(  \alpha,T\right)  \zeta_{\alpha,T}%
-\mathcal{E}_{+}\left(  s_{i}\alpha,T\right)  \zeta_{s_{i}\alpha,T}\right)  .
\]
Suppose $\alpha\in A$ then $r\left(  i+1,\alpha\right)  =r\left(
i,\alpha\right)  +1$ thus the values $r\left(  i,\alpha\right)  $ and
$r\left(  i+1,\alpha\right)  $ belong to some interval $\left[  a_{j}%
,b_{j}\right]  $ (where $\mathcal{S}_{\left[  a_{j},b_{j}\right]  }$ is a
factor of $W_{\lambda}$) and are adjacent entries in some column of $T$, hence
$s_{i}\zeta_{\alpha,T}=-\zeta_{\alpha,T}$. Next let $\alpha\in B$ then the
corresponding term in the sum is $\left(  -1\right)  ^{\mathrm{inv}\left(
\alpha\right)  }\mathcal{E}_{+}\left(  \alpha,T\right)  \left(  \zeta
_{\alpha,T}-\frac{\mathcal{E}_{+}\left(  s_{i}\alpha,T\right)  }%
{\mathcal{E}_{+}\left(  \alpha,T\right)  }\zeta_{s_{i}\alpha,T}\right)  $, a
scalar multiple of $\zeta_{\alpha,T}-\left(  1+b_{i}\left(  \alpha,T\right)
\right)  \zeta_{s_{i}\alpha,T}$, by an argument similar to the previous
theorem. This term satisfies $s_{i}f=-f$. Thus $s_{i}f_{\lambda,T}%
^{a}=-f_{\lambda,T}^{a}$. Consider $g=\sum_{w\in\mathcal{S}_{N}}$%
\textrm{sgn}$\left(  w\right)  w\zeta_{\lambda^{R},T}$; since $g$ is
$\mathcal{S}_{N}$-alternating it must equal a constant multiple $\gamma$ of
$f_{\lambda,T}.$ To find $\gamma$ consider the coefficients of $x^{\lambda
}v_{T}$ in $f_{\lambda,T}$ and $g$. The coefficient in $f_{\lambda,T}$ is $1$
(by definition of $\zeta_{\lambda,T}$). The term $x^{\lambda}v_{T}$ appears in
$w\zeta_{\lambda^{R},T}$ exactly when $w=w_{1}w_{\lambda^{R}}^{-1}$ for
$w_{1}\in W_{\lambda}$. Let $\varepsilon=$\textrm{sgn}$\left(  w_{\lambda^{R}%
}\right)  =\left(  -1\right)  ^{\mathrm{inv}\left(  \lambda^{R}\right)  }$,
because the length of $w_{\lambda^{R}}$ is \textrm{inv}$\left(  \lambda
^{R}\right)  $). Furthermore%
\begin{align*}
\mathrm{sgn}\left(  w_{1}w_{\lambda^{R}}^{-1}\right)  w_{1}w_{\lambda^{R}%
}^{-1}\zeta_{\lambda^{R},T}  &  =\mathrm{sgn}\left(  w_{1}w_{\lambda^{R}}%
^{-1}\right)  w_{1}w_{\lambda^{R}}^{-1}\left(  x^{\lambda^{R}}w_{\lambda^{R}%
}v_{T}\right)  +h_{1}\\
&  =\varepsilon~\mathrm{sgn}\left(  w_{1}\right)  w_{1}\left(  x^{\lambda
}v_{T}\right)  +h_{2}\\
&  =\varepsilon x^{\lambda}v_{T}+h_{2}%
\end{align*}
where $h_{1}$ and $h_{2}$ are terms of lower order, that is, of the form
$\sum_{\beta\vartriangleleft\lambda}x^{\beta}u_{\beta}$ with $u_{\beta}\in
V_{\tau}$. Thus $g=\varepsilon\left(  \#W_{\lambda}\right)  f_{\lambda,T}$ and%
\begin{align*}
\left\Vert f_{\lambda,T}^{a}\right\Vert ^{2}  &  =\frac{\varepsilon
}{\#W_{\lambda}}\left\langle g,f_{\lambda,T}\right\rangle =\frac{\varepsilon
}{\#W_{\lambda}}\sum_{w\in\mathcal{S}_{N}}\mathrm{sgn}\left(  w\right)
\left\langle w\zeta_{\lambda^{R},T},f_{\lambda,T}^{s}\right\rangle \\
&  =\frac{\varepsilon}{\#W_{\lambda}}\sum_{w\in\mathcal{S}_{N}}\mathrm{sgn}%
\left(  w\right)  \left\langle \zeta_{\lambda^{R},T},w^{-1}f_{\lambda,T}%
^{s}\right\rangle =\frac{\varepsilon N!}{\#W_{\lambda}}\left\langle
\zeta_{\lambda^{R},T},f_{\lambda,T}^{s}\right\rangle \\
&  =\frac{\varepsilon N!}{\#W_{\lambda}}\left(  -1\right)  ^{\mathrm{inv}%
\left(  \lambda^{R}\right)  }\mathcal{E}_{+}\left(  \lambda^{R},T\right)
\left\Vert \zeta_{\lambda^{R},T}\right\Vert ^{2}=\frac{N!\mathcal{E}%
_{+}\left(  \lambda^{R},T\right)  }{\left(  \#W_{\lambda}\right)
\mathcal{E}_{2}\left(  \lambda^{R},T\right)  }\left\Vert \zeta_{\lambda
,T}\right\Vert ^{2},
\end{align*}
This completes the proof.
\end{proof}

For the general case we introduce the following:

\begin{definition}
For $\lambda\in\mathbb{N}_{0}^{N,+}$ and $T\in Y\left(  \tau\right)  $ define
the tableau $\left\lfloor \lambda,T\right\rfloor $ to be the assignment of
$\lambda_{1},\lambda_{2},\ldots,\lambda_{N}$ to the nodes of the Ferrers
diagram of $\tau$ so that the entry at $T\left(  i\right)  $ is $\lambda
_{i},i\in\left[  1,N\right]  $. Thus the entries of $\left\langle
\lambda,T\right\rangle $ are weakly increasing ($\leq$) in each row and in
each column. The set of $T^{\prime}$ satisfying $\left\lfloor \lambda
,T^{\prime}\right\rfloor =\left\lfloor \lambda,T\right\rfloor $ is exactly
$Y\left(  T;W_{\lambda}\right)  $.
\end{definition}

Let $T_{0}\in Y\left(  \tau\right)  $ such that $T_{0}$ satisfies condition
$\left[  a_{i},b_{i}\right]  _{\mathrm{cm}}$ for each factor $\mathcal{S}%
_{\left[  a_{i},b_{i}\right]  }$ of $W_{\lambda}$ and $a_{i}\leq j_{1}%
<j_{2}\leq b_{i}$ implies \textrm{cm}$\left(  j_{1},T_{0}\right)
>$\textrm{cm}$\left(  j_{2},T_{0}\right)  $. This condition is equivalent to
the tableau $\left\lfloor \lambda,T_{0}\right\rfloor $ being column-strict
(the entries strictly increase in each column, see \cite[p.5]{Md}, such
tableaux are also called semistandard Young tableaux) and $T_{0}$ has a
certain extremal property among all $T\in Y\left(  T_{0};W_{\lambda}\right)
$. Let%

\begin{align*}
f_{\lambda,T_{0}}^{s}  &  =\sum_{\alpha^{+}=\lambda}\sum_{T\in Y\left(
T_{0};W_{\lambda}\right)  }\prod\limits_{j=1}^{n}P_{0}\left(  T;a_{j}%
,b_{j}\right)  \mathcal{E}_{-}\left(  \alpha,T\right)  \zeta_{\alpha,T},\\
u_{\lambda,T_{0}}  &  =\sum_{T\in Y\left(  T_{0};W_{\lambda}\right)  }%
\prod\limits_{j=1}^{n}P_{0}\left(  T;a_{j},b_{j}\right)  v_{T}\in V_{\tau}.
\end{align*}
The term involving $x^{\lambda}$ is $h_{0}=\sum\limits_{T\in Y\left(
T_{0};W_{\lambda}\right)  }\prod\limits_{j=1}^{n}P_{0}\left(  T;a_{j}%
,b_{j}\right)  \zeta_{\lambda,T}$, thus the leading term in $f_{\lambda,T_{0}%
}^{s}$ is $x^{\lambda}u_{\lambda,T_{0}}$. From the transformation rules in
Proposition \ref{ai=ai1} it follows that $\left\Vert h_{0}\right\Vert
^{2}=\left\Vert \zeta_{\lambda,T_{0}}\right\Vert ^{2}\left\Vert u_{\lambda
,T_{0}}\right\Vert _{0}^{2}/\left\Vert v_{T_{0}}\right\Vert _{0}^{2}$ (see
Corollary \ref{svnorm}). Also $h_{0}$ is $W_{\lambda}$-invariant. In the
symbol $f_{\lambda,T_{0}}^{a}$ one could replace $T_{0}$ by any $T\in Y\left(
T_{0};W_{\lambda}\right)  $; then $T_{0}$ is the unique solution of
$\prod_{j=1}^{n}P_{0}\left(  T;a_{j},b_{j}\right)  =1$.

\begin{theorem}
\label{fsT0}$wf_{\lambda,T_{0}}^{s}=f_{\lambda,T_{0}}^{s}$ for all
$w\in\mathcal{S}_{N}$ and%
\[
\left\Vert f_{\lambda,T_{0}}^{s}\right\Vert ^{2}=\frac{N!}{\#W_{\lambda}}%
\frac{\left\Vert u_{\lambda,T_{0}}\right\Vert _{0}^{2}}{\mathcal{E}_{+}\left(
\lambda^{R},T_{0}\right)  \left\Vert v_{T_{0}}\right\Vert _{0}^{2}}\left\Vert
\zeta_{\lambda,T_{0}}\right\Vert ^{2}.
\]

\end{theorem}

\begin{proof}
Let $\mathcal{F}\left(  \alpha,T\right)  =\prod_{j=1}^{n}P_{0}\left(
T;a_{j},b_{j}\right)  \mathcal{E}_{-}\left(  \alpha,T\right)  $. Fix
$i\in\left[  1,N-1\right]  $ and collect the terms of $f_{\lambda,T_{0}}^{s}$
into three parts. Let
\begin{align*}
L  &  =\left\{  \left(  \alpha,T\right)  :\alpha^{+}=\lambda,T\in Y\left(
T_{0};W_{\lambda}\right)  \right\} \\
A  &  =\left\{  \left(  \alpha,T\right)  \in L:\alpha_{i}=\alpha
_{i+1},\mathrm{rw}\left(  r\left(  \alpha,i\right)  ,T\right)  =\mathrm{rw}%
\left(  r\left(  \alpha,i\right)  +1,T\right)  \right\}  ,\\
B  &  =\left\{  \left(  \alpha,T\right)  \in L:\alpha_{i}>\alpha
_{i+1}\right\}  ,\\
C  &  =\left\{  \left(  \alpha,T\right)  \in L:\alpha_{i}=\alpha
_{i+1},\mathrm{rw}\left(  r\left(  \alpha,i\right)  ,T\right)  <\mathrm{rw}%
\left(  r\left(  \alpha,i\right)  +1,T\right)  \right\}  .
\end{align*}
The first part is $\sum_{\left(  \alpha,T\right)  \in A}\mathcal{F}\left(
\alpha,T\right)  \zeta_{\alpha,T}$ and in this case $s_{i}\zeta_{\alpha
,T}=\zeta_{\alpha,T}$. The second part is%
\begin{align*}
&  \sum_{\left(  \alpha,T\right)  \in B}\left(  \mathcal{F}\left(
\alpha,T\right)  \zeta_{\alpha,T}+\mathcal{F}\left(  s_{i}\alpha,T\right)
\zeta_{s_{i}\alpha,T}\right) \\
&  =\sum_{\left(  \alpha,T\right)  \in B}\mathcal{F}\left(  \alpha,T\right)
\left(  \zeta_{\alpha,T}+\frac{\mathcal{F}\left(  s_{i}\alpha,T\right)
}{\mathcal{F}\left(  \alpha,T\right)  }\zeta_{s_{i}\alpha,T}\right)  .
\end{align*}
Just as in Proposition \ref{vsym} $\frac{\mathcal{F}\left(  s_{i}%
\alpha,T\right)  }{\mathcal{F}\left(  \alpha,T\right)  }=1-b_{i}\left(
\alpha,T\right)  $, and hence this sum is $s_{i}$-invariant. For use in $C$
let $I\left(  \alpha\right)  =$\textrm{rw}$\left(  \alpha,i\right)  $. Then
the third part is%
\begin{align*}
&  \sum_{\left(  \alpha,T\right)  \in C}\left(  \mathcal{F}\left(
\alpha,T\right)  \zeta_{\alpha,T}+\mathcal{F}\left(  \alpha,s_{I\left(
\alpha\right)  }T\right)  \zeta_{\alpha,s_{I\left(  \alpha\right)  }T}\right)
\\
&  =\sum_{\left(  \alpha,T\right)  \in C}\mathcal{F}\left(  \alpha,s_{I\left(
\alpha\right)  }T\right)  \left(  \frac{\mathcal{F}\left(  \alpha,T\right)
}{\mathcal{F}\left(  \alpha,s_{I\left(  \alpha\right)  }T\right)  }%
\zeta_{\alpha,T}+\zeta_{\alpha,s_{I\left(  \alpha\right)  }T}\right)  .
\end{align*}
To show that each term is $s_{i}$-invariant we must show $\frac{\mathcal{F}%
\left(  \alpha,T\right)  }{\mathcal{F}\left(  \alpha,s_{I\left(
\alpha\right)  }T\right)  }=b_{I\left(  \alpha\right)  }+1$. Fix such a term.
The equality $\alpha_{i}=\alpha_{i+1}$ implies $\left[  I\left(
\alpha\right)  ,I\left(  \alpha\right)  +1\right]  \subset\left[  a_{i}%
,b_{i}\right]  $ for some $i$. Thus%
\[
\frac{\prod_{j=1}^{n}P_{0}\left(  T;a_{j},b_{j}\right)  }{\prod_{j=1}^{n}%
P_{0}\left(  s_{I\left(  \alpha\right)  }T;a_{j},b_{j}\right)  }=\frac
{P_{0}\left(  T;a_{i},b_{i}\right)  }{P_{0}\left(  s_{I\left(  \alpha\right)
}T;a_{i},b_{i}\right)  }=1+b_{I\left(  \alpha\right)  }\left(  T\right)  .
\]
Finally consider $\mathcal{E}_{-}\left(  \alpha,T\right)  /\mathcal{E}%
_{-}\left(  \alpha,s_{I\left(  \alpha\right)  }T\right)  $; let $g_{lj}\left(
T\right)  =1-\frac{\kappa}{\alpha_{j}-\alpha_{l}+\kappa\left(  c\left(
r\left(  \alpha,j\right)  ,T\right)  -c\left(  r\left(  \alpha,l\right)
,T\right)  \right)  }$ if $l<j$ and $\alpha_{l}<\alpha_{j}$, and
$g_{lj}\left(  T\right)  =1$ otherwise. Then $c\left(  r\left(  \alpha
,j\right)  ,T\right)  =c\left(  r\left(  \alpha,j\right)  ,s_{I\left(
\alpha\right)  }T\right)  $ whenever $r\left(  \alpha,j\right)  \notin\left\{
I\left(  \alpha\right)  ,I\left(  \alpha\right)  +1\right\}  $, also $c\left(
r\left(  \alpha,i\right)  ,T\right)  =c\left(  r\left(  \alpha,i+1\right)
,s_{I\left(  \alpha\right)  }T\right)  $ and $c\left(  r\left(  \alpha
,i+1\right)  ,T\right)  =c\left(  r\left(  \alpha,i\right)  ,s_{I\left(
\alpha\right)  }T\right)  $. Thus $g_{l,i}\left(  T\right)  =g_{l,i+1}\left(
s_{I\left(  \alpha\right)  }T\right)  $ and $g_{l,i}\left(  s_{I\left(
\alpha\right)  }T\right)  =g_{l,i+1}\left(  T\right)  $ for $1\leq l<i$ with
similar relations for $g_{ij}$ and $g_{i+1,j}$ when $i+1<j\leq N$. Also
$g_{i,i+1}\left(  T\right)  =1=g_{i,i+1}\left(  s_{I\left(  \alpha\right)
}T\right)  $ thus $\mathcal{E}_{-}\left(  \alpha,T\right)  =\prod_{1\leq
l<j\leq N}g_{lj}\left(  T\right)  =\mathcal{E}_{-}\left(  \alpha,s_{I\left(
\alpha\right)  }T\right)  $. Hence $s_{i}f_{\lambda,T_{0}}^{s}=f_{\lambda
,T_{0}}^{s}$.

To compute $\left\Vert f_{\lambda,T_{0}}^{s}\right\Vert ^{2}$ consider
$\sum\limits_{w\in\mathcal{S}_{N}}wh_{R}$ where $h_{R}=\sum\limits_{T\in
Y\left(  T_{0};W_{\lambda}\right)  }\prod\limits_{j=1}^{n}P_{0}\left(
T;a_{j},b_{j}\right)  \zeta_{\lambda^{R},T}$. By the argument used above for
type (4) $\mathcal{E}_{-}\left(  \lambda^{R},T\right)  =\mathcal{E}_{-}\left(
\lambda^{R},T_{0}\right)  $ for all $T\in Y\left(  T_{0};W_{\lambda}\right)
$. Thus the term for $\alpha=\lambda^{R}$ in $f_{\lambda,T_{0}}^{s}$ is
$\mathcal{E}_{-}\left(  \lambda^{R},T_{0}\right)  h_{R}$ and leading term in
$h_{R}$ is $x^{\lambda^{R}}w_{\lambda^{R}}u_{\lambda,T_{0}}$. Similarly to the
proof of Theorem \ref{symz} we conclude $\sum\limits_{w\in\mathcal{S}_{N}%
}wh_{R}=\left(  \#W_{\lambda}\right)  f_{\lambda,T_{0}}^{s}$ and $\left(
\#W_{\lambda}\right)  \left\Vert f_{\lambda,T_{0}}^{s}\right\Vert
^{2}=N!\left\langle h_{R},f_{\lambda,T_{0}}^{s}\right\rangle =N!\mathcal{E}%
_{-}\left(  \lambda^{R},T_{0}\right)  \left\Vert g\right\Vert ^{2}$. Finally
$\left\Vert g\right\Vert ^{2}=\left\Vert \zeta_{\lambda^{R},T_{0}}\right\Vert
^{2}\left\Vert u_{\lambda,T_{0}}\right\Vert _{0}^{2}/\left\Vert v_{T_{0}%
}\right\Vert _{0}^{2}=\dfrac{\left\Vert \zeta_{\lambda,T_{0}}\right\Vert
^{2}\left\Vert u_{\lambda,T_{0}}\right\Vert _{0}^{2}}{\mathcal{E}_{2}\left(
\lambda^{R},T_{0}\right)  \left\Vert v_{T_{0}}\right\Vert _{0}^{2}}$.
\end{proof}

\begin{corollary}
Suppose $\lambda,\mu\in\mathbb{N}_{0}^{N,+}$ and $T_{1},T_{2}\in Y\left(
\tau\right)  $ such that $\left\lfloor \lambda,T_{1}\right\rfloor $ and
$\left\lfloor \mu,T_{2}\right\rfloor $ are column-strict. If $\lambda\neq\mu$
or $T_{2}\notin Y\left(  T_{1};W_{\lambda}\right)  $ then $\left\langle
f_{\lambda,T_{1}}^{s},f_{\mu,T_{2}}^{s}\right\rangle =0$.
\end{corollary}

Let $T_{0}\in Y\left(  \tau\right)  $ such that $T_{0}$ satisfies condition
$\left[  a_{i},b_{i}\right]  _{\mathrm{rw}}$ for each factor $\mathcal{S}%
_{\left[  a_{i},b_{i}\right]  }$ of $W_{\lambda}$ and $a_{i}\leq j_{1}%
<j_{2}\leq b_{i}$ implies \textrm{cm}$\left(  j_{1},T_{0}\right)  \leq
$\textrm{cm}$\left(  j_{2},T_{0}\right)  $. This condition is equivalent to
the tableau $\left\lfloor \lambda,T\right\rfloor $ being row-strict (the
entries strictly increase in each row), and $T_{0}$ having a certain extremal
property. Let%
\begin{align*}
f_{\lambda,T_{0}}^{a}  &  =\sum_{\alpha^{+}=\lambda}\left(  -1\right)
^{\mathrm{inv}\left(  \alpha\right)  }\sum_{T\in Y\left(  T_{0};W_{\lambda
}\right)  }\prod\limits_{j=1}^{n}P_{1}\left(  T;a_{j},b_{j}\right)
\mathcal{E}_{+}\left(  \alpha,T\right)  \zeta_{\alpha,T},\\
u_{\lambda,T_{0}}  &  =\sum_{T\in Y\left(  T_{0};W_{\lambda}\right)  }%
\prod\limits_{j=1}^{n}P_{1}\left(  T;a_{j},b_{j}\right)  v_{T}\in V_{\tau}.
\end{align*}
The term involving $x^{\lambda}$ is $h_{0}=\sum\limits_{T\in Y\left(
T_{0};W_{\lambda}\right)  }\prod\limits_{j=1}^{n}P_{1}\left(  T;a_{j}%
,b_{j}\right)  \zeta_{\lambda,T}$, thus the leading term in $f_{\lambda,T_{0}%
}^{a}$ is $x^{\lambda}u_{\lambda,T_{0}}$. From the transformation rules in
Proposition \ref{ai=ai1} it follows that $\left\Vert h_{0}\right\Vert
^{2}=\left\Vert \zeta_{\lambda,T_{0}}\right\Vert ^{2}\left\Vert u_{\lambda
,T_{0}}\right\Vert ^{2}/\left\Vert v_{T_{0}}\right\Vert ^{2}$ (see Proposition
\ref{avnorm}). Also $h_{0}$ is $W_{\lambda}$-antisymmetric.

\begin{theorem}
\label{faT0}$wf_{\lambda,T_{0}}^{a}=$\textrm{sgn}$\left(  w\right)
f_{\lambda,T_{0}}^{a}$ for all $w\in\mathcal{S}_{N}$ and%
\[
\left\Vert f_{\lambda,T_{0}}^{a}\right\Vert ^{2}=\frac{N!}{\#W_{\lambda}}%
\frac{\left\Vert u_{\lambda,T_{0}}\right\Vert _{0}^{2}}{\mathcal{E}_{-}\left(
\lambda^{R},T_{0}\right)  \left\Vert v_{T_{0}}\right\Vert _{0}^{2}}\left\Vert
\zeta_{\lambda,T_{0}}\right\Vert ^{2}.
\]

\end{theorem}

\begin{proof}
Let $\mathcal{F}\left(  \alpha,T\right)  =\prod_{j=1}^{n}P_{1}\left(
T;a_{j},b_{j}\right)  \mathcal{E}_{+}\left(  \alpha,T\right)  $. Fix
$i\in\left[  1,N-1\right]  $ and collect the terms of $f_{\lambda,T_{0}}^{a}$
into three parts. Let
\begin{align*}
L  &  =\left\{  \left(  \alpha,T\right)  :\alpha^{+}=\lambda,T\in Y\left(
T_{0};W_{\lambda}\right)  \right\} \\
A  &  =\left\{  \left(  \alpha,T\right)  \in L:\alpha_{i}=\alpha
_{i+1},\mathrm{cm}\left(  r\left(  \alpha,i\right)  ,T\right)  =\mathrm{cm}%
\left(  r\left(  \alpha,i\right)  +1,T\right)  \right\}  ,\\
B  &  =\left\{  \left(  \alpha,T\right)  \in L:\alpha_{i}>\alpha
_{i+1}\right\}  ,\\
C  &  =\left\{  \left(  \alpha,T\right)  \in L:\alpha_{i}=\alpha
_{i+1},\mathrm{rw}\left(  r\left(  \alpha,i\right)  ,T\right)  <\mathrm{rw}%
\left(  r\left(  \alpha,i\right)  +1,T\right)  \right\}  .
\end{align*}

\end{proof}

The proof that each of the following satisfies $s_{i}f=-f$ is analogous to the
proof of the previous theorem:%
\begin{align*}
&  \sum_{\left(  \alpha,T\right)  \in A}\left(  -1\right)  ^{\mathrm{inv}%
\left(  \alpha\right)  }\mathcal{F}\left(  \alpha,T\right)  \zeta_{\alpha
,T},\\
&  \sum_{\left(  \alpha,T\right)  \in B}\left(  -1\right)  ^{\mathrm{inv}%
\left(  \alpha\right)  }\left(  \mathcal{F}\left(  \alpha,T\right)
\zeta_{\alpha,T}-\mathcal{F}\left(  s_{i}\alpha,T\right)  \zeta_{s_{i}%
\alpha,T}\right) \\
&  =\sum_{\left(  \alpha,T\right)  \in B}\left(  -1\right)  ^{\mathrm{inv}%
\left(  \alpha\right)  }\mathcal{F}\left(  \alpha,T\right)  \left(
\zeta_{\alpha,T}-\frac{\mathcal{F}\left(  s_{i}\alpha,T\right)  }%
{\mathcal{F}\left(  \alpha,T\right)  }\zeta_{s_{i}\alpha,T}\right)  ,\\
&  \sum_{\left(  \alpha,T\right)  \in C}\left(  -1\right)  ^{\mathrm{inv}%
\left(  \alpha\right)  }\left(  \mathcal{F}\left(  \alpha,T\right)
\zeta_{\alpha,T}+\mathcal{F}\left(  \alpha,s_{I\left(  \alpha\right)
}T\right)  \zeta_{\alpha,s_{I\left(  \alpha\right)  }T}\right) \\
&  =\sum_{\left(  \alpha,T\right)  \in C}\left(  -1\right)  ^{\mathrm{inv}%
\left(  \alpha\right)  }\mathcal{F}\left(  \alpha,s_{I\left(  \alpha\right)
}T\right)  \left(  \frac{\mathcal{F}\left(  \alpha,T\right)  }{\mathcal{F}%
\left(  \alpha,s_{I\left(  \alpha\right)  }T\right)  }\zeta_{\alpha,T}%
+\zeta_{\alpha,s_{I\left(  \alpha\right)  }T}\right)  .
\end{align*}
In the second equation$\frac{\mathcal{F}\left(  s_{i}\alpha,T\right)
}{\mathcal{F}\left(  \alpha,T\right)  }=1+b_{i}\left(  \alpha,T\right)  $. In
the third equation $I=r\left(  \alpha,i\right)  $ and $\frac{\mathcal{F}%
\left(  \alpha,T\right)  }{\mathcal{F}\left(  \alpha,s_{I\left(
\alpha\right)  }T\right)  }=b_{I\left(  \alpha\right)  }-1$. The proof for the
norm formula is also analogous, based on $\sum_{w\in\mathcal{S}_{N}}wh_{R}$
where $h_{R}=\sum\limits_{T\in Y\left(  T_{0};W_{\lambda}\right)  }%
\prod\limits_{j=1}^{n}P_{1}\left(  T;a_{j},b_{j}\right)  \zeta_{\lambda^{R}%
,T}$. Note \textrm{sgn}$\left(  w_{\lambda^{R}}\right)  =\left(  -1\right)
^{\mathrm{inv}\left(  \lambda^{R}\right)  }$.

\begin{remark}
The polynomials in Theorems \ref{fsT0} and \ref{faT0} form orthogonal bases
for the symmetric and antisymmetric polynomials, respectively, in $M\left(
\tau\right)  $.
\end{remark}

We now establish the striking results concerning the norms of certain
symmetric and antisymmetric polynomials. For a given partition $\tau$ of $N$
there are unique symmetric and antisymmetric polynomials of minimum degree in
the standard module $\mathcal{M}\left(  \tau\right)  $. It is obvious that the
column-strict tableau $\left\langle \lambda,T\right\rangle $ with minimum
$\left\vert \lambda\right\vert $ has the entries $0$ in row \#1, $1$ in row
\#2 and so on (consider the minimum entries in each column). Denote this
partition by $\delta^{s}\left(  \tau\right)  $ and the unique possible $T$ by
$T^{s}$ (the entries $N,N-1,\ldots,2,1$ are entered row-by-row in the Ferrers
diagram of $\tau$). Example: let $\tau=\left(  5,3,2\right)  $ then
\[
T^{s}=%
\begin{array}
[c]{ccccc}%
10 & 9 & 8 & 7 & 6\\
5 & 4 & 3 &  & \\
2 & 1 &  &  &
\end{array}
,\left\lfloor \delta^{s}\left(  \tau\right)  ,T^{s}\right\rfloor =%
\begin{array}
[c]{ccccc}%
0 & 0 & 0 & 0 & 0\\
1 & 1 & 1 &  & \\
2 & 2 &  &  &
\end{array}
,
\]
and $\delta^{s}\left(  \tau\right)  =\left(  2,2,1,1,1,0,0,0,0,0\right)  $.

Similarly the row-strict tableau $\left\langle \lambda,T\right\rangle $ with
minimum $\left\vert \lambda\right\vert $ has the entries $0$ in column \#1,
$1$ in column \#2 and so on (consider the minimum entries in each row). Denote
this partition by $\delta^{a}\left(  \tau\right)  $ and the unique possible
$T$ by $T^{a}$ (the entries $N,N-1,\ldots,2,1$ are entered column-by-column in
the Ferrers diagram of $\tau$). Example: let $\tau=\left(  5,3,2\right)  $
then%
\[
T^{a}=%
\begin{array}
[c]{ccccc}%
10 & 7 & 4 & 2 & 1\\
9 & 6 & 3 &  & \\
8 & 5 &  &  &
\end{array}
,\left\lfloor \delta^{a}\left(  \tau\right)  ,T^{a}\right\rfloor =%
\begin{array}
[c]{ccccc}%
0 & 1 & 2 & 3 & 4\\
0 & 1 & 2 &  & \\
0 & 1 &  &  &
\end{array}
,
\]
and $\delta^{a}\left(  \tau\right)  =\left(  4,3,2,2,1,1,1,0,0,0\right)  $.
The sum of the hook-lengths of $\tau$ equals $\left\vert \delta^{s}\left(
\tau\right)  \right\vert +\left\vert \delta^{a}\left(  \tau\right)
\right\vert +N$ (see \cite[Ex.2, p.11]{Md}).

Let $f_{\tau}^{s}=f_{\delta^{s}\left(  \tau\right)  ,T^{s}}^{s}$ and $f_{\tau
}^{a}=f_{\delta^{a}\left(  \tau\right)  ,T^{a}}^{a}$. These polynomials are
actually independent of $\kappa$; there is no composition $\alpha$ such that
$\alpha\vartriangleleft\delta^{s}\left(  \tau\right)  $ and $\alpha^{+}%
\neq\delta^{s}\left(  \tau\right)  $ which can occur in a symmetric
polynomial, due to the minimality of $\delta^{s}\left(  \tau\right)  $. A
similar argument applies to $\delta^{a}\left(  \tau\right)  $. To compute the
norms of $\left\Vert f_{\tau}^{s}\right\Vert ^{2}$ and $\left\Vert f_{\tau
}^{a}\right\Vert ^{2}$ we use the special properties of $\delta^{s}\left(
\tau\right)  $ to write simplified formulae. To use the formulae in Theorems
\ref{normz} and \ref{symz} note that $\delta_{a}\left(  \tau\right)  _{j}=i-1$
when $j$ appears in row $\#i$ of $T^{s}$, and the corresponding contents of
$T^{s}$ are $1-i,\ldots,\tau_{i}-i$. Let $L=\ell\left(  \tau\right)  $ and%
\begin{align*}
P_{1}\left(  \tau\right)   &  =\prod_{i=2}^{L}\prod_{j=1}^{\tau_{i}}\left(
1+\kappa\left(  j-i\right)  \right)  _{i-1},\\
P_{2}\left(  \tau\right)   &  =\prod_{1\leq i<j\leq L}\prod_{j_{1}=1}%
^{\tau_{i}}\prod_{j_{2}=1}^{\tau_{j}}\prod_{r=1}^{j-i}\left(  1-\frac
{\kappa^{2}}{\left(  r+\kappa\left(  j_{2}-j_{1}-j+i\right)  \right)  ^{2}%
}\right)  ,\\
P_{3}\left(  \tau\right)   &  =\prod_{1\leq i<j\leq L}\prod_{j_{1}=1}%
^{\tau_{i}}\prod_{j_{2}=1}^{\tau_{j}}\left(  1+\frac{\kappa}{j-i+\kappa\left(
j_{2}-j_{1}-j+i\right)  }\right)  ;\\
H^{s}\left(  \tau\right)   &  =\prod_{\left(  i,j\right)  \in\tau}\left(
1-\kappa h\left(  i,j\right)  \right)  _{\mathrm{leg}\left(  i,j\right)  }.
\end{align*}
Then $\left\Vert \zeta_{\delta^{s}\left(  \tau\right)  ,T^{s}}\right\Vert
^{2}=\left\Vert v_{T^{s}}\right\Vert ^{2}P_{1}\left(  \tau\right)
P_{2}\left(  \tau\right)  $ and $\mathcal{E}_{+}\left(  \delta^{s}\left(
\tau\right)  ^{R},T^{s}\right)  =P_{3}\left(  \tau\right)  $.

\begin{theorem}
Suppose $\tau$ is a partition then $\dfrac{P_{1}\left(  \tau\right)
P_{2}\left(  \tau\right)  }{P_{3}\left(  \tau\right)  }=H^{s}\left(
\tau\right)  $.
\end{theorem}

\begin{proof}
We use induction on the last part $\tau_{L}$. The induction begins with
$\tau=\left(  N\right)  $ where each product equals $1$. Let $\sigma=\left(
\tau_{1},\tau_{2},\ldots,\tau_{L}-1\right)  $ and assume the formula is valid
for $\sigma$. (It is possible that $\tau_{L}=1$ and $\ell\left(
\sigma\right)  =L-1$). The nodes in $\sigma$ and $\tau$ have the same
hook-lengths except for the nodes $\left(  i,L\right)  $ with $1\leq i<L$ and
$\left(  L,j\right)  $ with $1\leq j\leq\tau_{L}$. The latter have zero
leg-length and do not contribute to $H^{s}\left(  \sigma\right)  $ or
$H^{s}\left(  \tau\right)  $. Then for $1\leq i<L$%
\begin{align*}
\mathrm{arm}\left(  i,L;\sigma\right)   &  =\mathrm{arm}\left(  i,L;\tau
\right)  =\tau_{i}-\tau_{L}\\
\mathrm{leg}\left(  i,L;\sigma\right)  +1  &  =\mathrm{leg}\left(
i,L;\tau\right)  =L-i\\
h\left(  i,L;\sigma\right)  +1  &  =h\left(  i,L;\tau\right)  =1+\tau_{i}%
-\tau_{L}+L-i,
\end{align*}
and thus%
\[
\frac{H^{s}\left(  \tau\right)  }{H^{s}\left(  \sigma\right)  }=\prod
_{i=1}^{L-1}\frac{\left(  1-\kappa\left(  \tau_{i}-\tau_{L}+L-i+1\right)
\right)  _{L-i}}{\left(  1-\kappa\left(  \tau_{i}-\tau_{L}+L-i\right)
\right)  _{L-i-1}}.
\]
Firstly,%
\[
\frac{P_{1}\left(  \tau\right)  }{P_{1}\left(  \sigma\right)  }=\left(
1+\kappa\left(  \tau_{L}-L\right)  \right)  _{L-1};
\]
secondly%
\begin{align*}
\frac{P_{2}\left(  \tau\right)  }{P_{2}\left(  \sigma\right)  }  &
=\prod_{i=1}^{L-1}p_{i}\left(  \tau\right)  ,\\
p_{i}\left(  \tau\right)   &  =\prod_{r=1}^{L-i}\prod_{j=1}^{\tau_{i}}%
\frac{\left(  r+\kappa\left(  \tau_{L}-L+i\right)  -\left(  j-1\right)
\kappa\right)  \left(  r+\kappa\left(  \tau_{L}-L+i\right)  -\left(
j+1\right)  \kappa\right)  }{\left(  r+\kappa\left(  \tau_{L}-L+i\right)
-j\kappa\right)  \left(  r+\kappa\left(  \tau_{L}-L+i\right)  -j\kappa\right)
}\\
&  =\prod_{r=1}^{L-i}\frac{\left(  r+\kappa\left(  \tau_{L}-L+i\right)
\right)  \left(  r+\kappa\left(  \tau_{L}-L+i\right)  -\left(  \tau
_{i}+1\right)  \kappa\right)  }{\left(  r+\kappa\left(  \tau_{L}-L+i\right)
-\kappa\right)  \left(  r+\kappa\left(  \tau_{L}-L+i\right)  -\tau_{i}%
\kappa\right)  }\\
&  =\frac{\left(  1+\kappa\left(  \tau_{L}-L+i\right)  \right)  _{L-i}\left(
1+\kappa\left(  \tau_{L}-\tau_{i}-L+i-1\right)  \right)  _{L-i}}{\left(
1+\kappa\left(  \tau_{L}-L+i-1\right)  \right)  _{L-i}\left(  1+\kappa\left(
\tau_{L}-\tau_{i}-L+i\right)  \right)  _{L-i}};
\end{align*}
a telescoping product argument is used to produce the third line from the
second. Thirdly,%
\begin{align*}
\frac{P_{3}\left(  \sigma\right)  }{P_{3}\left(  \tau\right)  }  &
=\prod_{i=1}^{L-1}\prod_{j_{1}=1}^{\tau_{i}}\left(  \frac{L-i+\kappa\left(
\tau_{L}-j_{1}-L+i\right)  }{L-i+\kappa\left(  \tau_{L}+1-j_{1}-L+i\right)
}\right) \\
&  =\prod_{i=1}^{L-1}\frac{L-i+\kappa\left(  \tau_{L}-\tau_{i}-L+i\right)
}{L-i+\kappa\left(  \tau_{L}-L+i\right)  }.
\end{align*}
Combining these products and by use of%
\begin{align*}
\frac{L-i+\kappa\left(  \tau_{L}-\tau_{i}-L+i\right)  }{\left(  1+\kappa
\left(  \tau_{L}-\tau_{i}-L+i\right)  \right)  _{L-i}}  &  =\frac{1}{\left(
1+\kappa\left(  \tau_{L}-\tau_{i}-L+i\right)  \right)  _{L-i}},\\
\frac{\left(  1+\kappa\left(  \tau_{L}-L+i\right)  \right)  _{L-i}}%
{L-i+\kappa\left(  \tau_{L}-L+i\right)  }  &  =\left(  1+\kappa\left(
\tau_{L}-L+i\right)  \right)  _{L-i-1},
\end{align*}
we obtain%
\begin{align*}
\dfrac{P_{1}\left(  \tau\right)  P_{2}\left(  \tau\right)  P_{3}\left(
\sigma\right)  H^{s}\left(  \sigma\right)  }{P_{1}\left(  \sigma\right)
P_{2}\left(  \sigma\right)  P_{3}\left(  \tau\right)  H^{s}\left(
\tau\right)  }  &  =\left(  1+\kappa\left(  \tau_{L}-L\right)  \right)
_{L-1}\prod_{i=1}^{L-1}\frac{\left(  1+\kappa\left(  \tau_{L}-L+i\right)
\right)  _{L-i-1}}{\left(  1+\kappa\left(  \tau_{L}-L+i-1\right)  \right)
_{L-i}}\\
&  =1.
\end{align*}
The last step is actually easy: replace $i$ by $i-1$ in the numerator (and now
$2\leq i\leq L$) and cancel. This completes the induction.
\end{proof}

\begin{theorem}
\label{hook0}Suppose $\tau$ is a partition of $N$ then%
\[
\left\Vert f_{\tau}^{s}\right\Vert ^{2}=\frac{N!}{\prod_{i=1}^{\tau_{1}}%
\tau_{i}^{\prime}!}\left\Vert v_{T^{s}}\right\Vert _{0}^{2}\prod_{\left(
i,j\right)  \in\tau}\left(  1-\kappa h\left(  i,j\right)  \right)
_{\mathrm{leg}\left(  i,j\right)  }.
\]

\end{theorem}

\begin{proof}
The formulae of Theorem \ref{fsT0} and the previous Theorem imply this result.
The stabilizer subgroup $W_{\delta^{s}\left(  \tau\right)  }$ is acts on the
columns of $\tau$ and has order $\tau_{1}^{\prime}!\tau_{2}^{\prime}!\ldots$.
\end{proof}

\begin{theorem}
Suppose $\tau$ is a partition of $N$ then%
\[
\left\Vert f_{\tau}^{a}\right\Vert ^{2}=\frac{N!}{\prod_{i=1}^{\ell\left(
\tau\right)  }\tau_{i}!}\left\Vert v_{T^{a}}\right\Vert _{0}^{2}\prod_{\left(
i,j\right)  \in\tau}\left(  1+\kappa h\left(  i,j\right)  \right)
_{\mathrm{arm}\left(  i,j\right)  }.
\]

\end{theorem}

\begin{proof}
Apply Theorem \ref{faT0} and the formula in Theorem \ref{hook0} to the
conjugate $\tau^{\prime}$ of $\tau$ and with $\kappa$ replaced by $-\kappa$.
Then $\mathrm{leg}\left(  j,i;\tau^{\prime}\right)  =\mathrm{arm}\left(
i,j;\tau\right)  $ for $\left(  i,j\right)  \in\tau$. Note however that
$\left\Vert v_{T^{s}}\right\Vert _{0}^{2}/\left\Vert v_{T^{a}}\right\Vert
_{0}^{2}$ is computed by use of Proposition \ref{avnorm}.
\end{proof}

As example we use $\tau=\left(  5,3,2\right)  $ again. The hook-lengths and
norms are
\begin{align*}
&
\begin{array}
[c]{ccccc}%
7 & 6 & 4 & 2 & 1\\
4 & 3 & 1 &  & \\
2 & 1 &  &  &
\end{array}
,\\
\left\Vert f_{\tau}^{s}\right\Vert ^{2} &  =c_{0}\left(  1-7\kappa\right)
_{2}\left(  1-6\kappa\right)  _{2}\left(  1-4\kappa\right)  ^{2}\left(
1-3\kappa\right)  ,\\
\left\Vert f_{\tau}^{a}\right\Vert ^{2} &  =c_{1}\left(  1+7\kappa\right)
_{4}\left(  1+6\kappa\right)  _{3}\left(  1+4\kappa\right)  _{2}^{2}\left(
1+3\kappa\right)  \left(  1+2\kappa\right)  ^{2}.
\end{align*}
Analogously to the $M_{\left(  N\right)  }$ (trivial representation) result,
each hook-length $m$ appears in $m-1$ factors $\left(  m\kappa+r\right)  $
involving each nonzero residue class $\operatorname{mod}m$. In the example for
$m=6$ we obtain $6\kappa-2,6\kappa-1,6\kappa+1,6\kappa+2,6\kappa+3$. We
conjecture that the singular values for $M\left(  \tau\right)  $ form a subset
of $\left\{  \frac{n}{m}:m=h\left(  i,j\right)  ,\left(  i,j\right)  \in
\tau,\frac{n}{m}\notin\mathbb{Z}\right\}  $ ($\kappa_{0}\in\mathbb{Q}$ is a
singular value if there exists nonzero $f\in M\left(  \tau\right)  $ such that
$\mathcal{D}_{i}\left(  \kappa_{0}\right)  f=0$ for all $i\in\left[
1,N\right]  $; that is, the generic $\kappa$ is specialized to $\kappa_{0}$;
the condition is equivalent to $J_{\kappa_{0}}\left(  \tau\right)  \neq\left(
0\right)  $). As yet there is insufficient evidence for speculation about any
further restrictions.

S. Griffeth (personal communication, gratefully acknowledged) points out that
Theorem \ref{hook0} provides a new proof for one of the parts of the
Gordon-Stafford Theorem \cite[Cor. 3.13]{GS}; another proof was found by
Bezrukavnikov and Etingof \cite[Cor. 4.2]{BE}; note that these papers use
$c=-\kappa$ as parameter. An \textit{aspherical }module of the rational
Cherednik algebra is one containing no nonzero $\mathcal{S}_{N}$-invariant. If
some quotient module of a standard module is aspherical for a numerical value
$\kappa_{0}$ of $\kappa$ then $\kappa_{0}$ is called an \textit{aspherical
value}. Theorem \ref{hook0} shows that any aspherical value is in $\left\{
\frac{m}{n}:1\leq m<n\leq N\right\}  $ (this is one component of the
Gordon-Stafford theorem, which deals with the problem of Morita equivalence of
rational Cherednik algebras for parameters $\kappa$ and $\kappa-1$). Suppose
$M_{0}$ is a proper submodule of $M\left(  \tau\right)  $ for $\kappa
=\kappa_{0}\in\mathbb{Q}$ (that is, a specific numerical value). This means
that $M_{0}$ is closed under multiplication by $x_{i}$ and the action of
$\mathcal{D}_{i}$ for $i\in\left[  1,N\right]  $ and under the action of
$\mathcal{S}_{N}$. Then $f\in M_{0}$ implies $\left\langle g,f\right\rangle
=0$ for all $g\in M\left(  \tau\right)  $ ($M_{0}$ is a submodule of the
radical $J_{\kappa_{0}}\left(  \tau\right)  $, the maximal submodule.).
Indeed, by the definition of the contravariant form, $\left\langle x^{\alpha
}u,f\right\rangle =\left\langle u,\mathcal{D}^{\alpha}f\left(  x\right)
|_{x=0}\right\rangle _{0}$ for $\alpha\in\mathbb{N}_{0}^{N},u\in V_{\tau}$
(and $\mathcal{D}^{\alpha}=\prod_{i=1}^{N}\mathcal{D}_{i}^{\alpha_{i}}$). If
$f\in\mathcal{P}_{n}\otimes V_{\tau}$ and $\left\vert \alpha\right\vert =n$
then $\mathcal{D}^{\alpha}f\left(  x\right)  \in V_{\tau}$. If also $f\in
M_{0}$ then $\mathcal{D}^{\alpha}f\left(  x\right)  =0$, or else
$M_{0}=M\left(  \tau\right)  $. If $M\left(  \tau\right)  /M_{0}$ is
aspherical then $f_{\tau}^{s}\in M_{0}$ and $\kappa=\kappa_{0}$ is a zero of
$\prod_{\left(  i,j\right)  \in\tau}\left(  1-\kappa h\left(  i,j\right)
\right)  _{\mathrm{leg}\left(  i,j\right)  }$.

\end{document}